\documentclass[12pt,a4paper,leqno]{amsart}
\usepackage{amsmath, amsfonts,mathrsfs,amssymb,amscd,graphicx,setspace}
\usepackage[mathscr]{eucal}
\usepackage{mathtools}
\usepackage{tdclock}
\usepackage{bm}
\usepackage{hyperref,url}
\usepackage{enumitem}
\usepackage{color,soul} 

\DeclareMathOperator{\Li}{Li}

\newcommand{\Zb}{\mathbb{Z}}
\newcommand{\Cb}{\mathbb{C}}
\newcommand{\z}{\zeta}

\newcommand{\onen}[1]{\{1\}^{#1}}

\makeatletter
\numberwithin{equation}{section}
\allowdisplaybreaks

\newtheorem{theorem}{Theorem}
\newtheorem{proposition}{Proposition}

\newtheorem{corollary}{Corollary}
\newtheorem{lemma}{Lemma}

\theoremstyle{remark}

\newtheorem{example}{Example}
\allowdisplaybreaks
\dedicatory{}
\begin{document}
	\initclock
	\title[]{Mordell-Tornheim multiple zeta-functions, their integral analogues, and relations among multiple polylogarithms}
	\author[Kohji Matsumoto]{Kohji Matsumoto}
	\address[Kohji Matsumoto]{Graduate School of Mathematics, Nagoya University, Furocho, Chikusa-ku, Nagoya 464-8602, Japan and Center for General Education, Aichi Institute of Technology, 1247 Yachigusa,
	Yakusa-cho, Toyota 470-0392, Japan}
	\email{kohjimat@math.nagoya-u.ac.jp}
	\author{Kazuhiro Onodera}
	\address[Kazuhiro Onodera]{Division of Mathematics, Chiba Institute of Technology,
	2-1-1 Shibazono, Narashino, Chiba 275-0023, Japan}
	\email{onodera@p.chibakoudai.jp}
	\author[Dilip K. Sahoo]{Dilip K. Sahoo}
	\address[Dilip K. Sahoo] {Harish-Chandra Research Institute Chhatnag Road, Jhunsi, Prayagraj(Allahabad) 211019, India}
	\email{mzfdilipmath@gmail.com}
	\begin{abstract}
We study the asymptotic behavior of a multiple series of Mordell-Tornheim type and its integral analogue. We show a relation between the multiple series and its integral analogue and to deeply investigate the behavior of the integral analogue. In the process of study, we obtain interesting nontrivial relations among multiple polylogarithms. Particularly we express certain finite linear combination of multiple polylogarithms in terms of logarithms and Riemann zeta values.
	\end{abstract}
	\maketitle
\section{Introduction and statement of results}

Let $r$ be a positive integer, $x$ be a real variable,
$a$ be a non-negative real number, 
$\bm{\omega}=(\omega_1,\omega_2,\ldots,\omega_r)$ be an $r$-tuple of positive real numbers,
$\bm{k}=(k_1,k_2,\ldots,k_r)$ be an $r$-tuple of positive integers,
and $\bm{z}=(z_1,z_2,\ldots,z_r)$ be an $r$-tuple of complex numbers with $|z_i|<1$.

The main objects of the present article are the multiple series
\begin{equation}
\label{function M}
M_r(x;\bm{\omega},a)
=\sum_{n_1,n_2,\ldots,n_r\ge 1}
\dfrac{1}{n_1n_2\cdots n_r(\omega_1n_1+\omega_2n_2+\cdots+\omega_r n_r+a)^x},
\end{equation}
its integral analogue
\begin{equation}
\label{def of I}
I_r(x;\bm{\omega}, a)
=\int_1^\infty \int_1^\infty \cdots \int_1^\infty
\dfrac{dt_1dt_2\cdots dt_r}{t_1t_2\cdots t_r(\omega_1 t_1+\omega_2t_2+\cdots +\omega_r t_r+a)^x},
\end{equation}
and the multiple polylogarithm defined as
\begin{equation}\label{def of MPL}
\Li_{\bm{k}} (\bm{z})
=\sum_{1\le  n_1<n_2<\cdots <n_r}
\dfrac{z_1^{n_1}z_2^{n_2}\cdots z_r^{n_r}}
{n_1^{k_1}n_2^{k_2}\cdots n_r^{k_r}}.
\end{equation}

We first study the behavior of
\eqref{function M} and \eqref{def of I} around $x=0$, and then discuss applications to
\eqref{def of MPL}.
As we will see later (Propositions \ref{convergence for I} and \ref{convergence}),
 \eqref{function M} and \eqref{def of I} are convergent only when $x>0$.
Therefore hereafter we assume $x>0$.

Historically, Mordell \cite{Mord58}
first introduced and studied the multiple series in the case when $x=1$ and $\bm{\omega}=(1,1,\ldots,1)$, 
and then Hoffman \cite{Hoff92} treated the case when $x$ is any positive integer, $\bm{\omega}=(1,1,\ldots,1)$, and $a=0$, under
the name of multiple harmonic series.    Lastly, combining this type of series by
Mordell and Hoffman with the double series
$\sum_{m,n\geq 1}m^{-u}n^{-v}(m+n)^{-w}$
($u,v,w$ are positive integers) studied by Tornheim \cite{Torn50}, 
the first-named author \cite{Mats03} introduced the multi-variable multiple zeta-function
\begin{align}
\label{MT-series}
\zeta_{MT,r}(s_1,\ldots,s_r,s_{r+1})=\sum_{n_1,\ldots,n_r\geq 1}
\frac{1}{n_1^{s_1}\cdots n_r^{s_r}(n_1+\cdots +n_r)^{s_{r+1}}},
\end{align}
where $s_1,\ldots,s_r,s_{r+1}$ are complex variables,
and called it the Mordell-Tornheim $r$-fold zeta-function.    Obviously,
$M_r(x;(1,1,\ldots,1),0)=\zeta_{MT,r}(1,\ldots,1,x)$.

In \cite{Mats03}, it has been proved that \eqref{MT-series} can be continued
meromorphically to the whole complex space $\mathbb{C}^{r+1}$ of dimension $r+1$.
Then various analytic properties of \eqref{MT-series} were established by several
subsequent papers, such as functional relations (Tsumura \cite{Tsum07},
Nakamura \cite{Naka06} and the second-named author \cite{Onod14}), the functional equation
(Okamoto and Onozuka \cite{OkOn16}), evaluation of special values at
non-positive integer points (Komori \cite{Komo08}, the second-named author \cite{Onod13}, \cite{Onod18},
\cite{Onod21}), and so on.

In \cite{Mats03}, possible singularities of \eqref{MT-series} have also
been listed up.    Moreover in \cite{MNOT08} it has been shown that all of those possible
singularities are indeed true singularities.    In particular, 
$s_1+s_2+\cdots+s_{r+1}=r$ is known to be a singular hyperplane.
Since the point $(1,1,\ldots,1,0)$ is on this hyperplane, the behavior of $M_r(x;(1,1,\ldots,1),0)$ 
should be singular when $x$ is close to $0$.
For example, in the simplest case,
\begin{equation}
\label{classical case}
M_1(x;(1),0)=\sum_{n=1}^\infty \dfrac{1}{n^{1+x}}=\dfrac{1}{x}+\gamma+O(x),
\end{equation}
where $\gamma$ is the Euler constant.
This is a classical result that can be shown by the Abel summation formula.
The main purpose of this article is to reveal the behavior of $M_r(x;\bm{\omega},a)$ around $x=0$ in the general case. As a special case of our result, we provide an asymptotic formula for $M_r(x;(1,1,\ldots,1),0)$ around $x=0$.

In this article, in the same way as the proof of \eqref{classical case}, 
we use the Abel summation formula to clarify the behavior of $M_r(x;\bm{\omega},a)$.
This naturally leads to integrals of the form \eqref{def of I}
and reduce our problem to understanding the behavior of $I_r(x;\bm{\omega},a)$ around $x=0$.
The first main result reveals this behavior up to the constant term.

\begin{theorem}
\label{asymptotic for I}
\begin{equation}
\label{eq:asymptotic for I}
I_r(x;\bm{\omega},a)
=\dfrac{e^{-\gamma x}}{\Gamma(x+1)}
\sum_{k=0}^r
(-1)^k\Lambda_k(\bm{\omega})
\dfrac{(r-k)!}{x^{r-k}}
+O(x),
\end{equation}
where
$\Gamma(x)$ is the gamma function,
$\Lambda_0(\bm{\omega})=1$, and
\[
\Lambda_k(\bm{\omega})=\sum_{1\le i_1<\cdots <i_k\le r}
\log \omega_{i_1} \cdots \log \omega_{i_k}
\]
for any positive integer $k$.
\end{theorem}
Since the factor $e^{-\gamma x}/\Gamma(x+1)$ can be explicitly expanded as a Taylor series around  $x=0$, the main term on the right-hand side of \eqref{eq:asymptotic for I} can be expressed
as a polynomial in $1/x$.
For instance,
\begin{align*}
I_1(x;\omega,a)
&=\dfrac1x-\log \omega +O(x),\\
I_2(x;\bm{\omega},a)
&=\dfrac{2}{x^2}-\dfrac{\log (\omega_1\omega_2)}{x}
+\log \omega_1 \log \omega_2 -\dfrac{\pi^2}{6}+O(x).
\end{align*}
Remark that $\gamma$ does not appear in these asymptotic formulas.
This  phenomenon also occurs in the general case,
as the coefficients of the Taylor expansion of $e^{-\gamma x}/\Gamma(x+1)$ do not contain $\gamma$ (see \eqref{power series for Gamma 1}).

An asymptotic formula of $M_r(x;\bm{\omega},a)$ can be derived from Theorem \ref{asymptotic for I}. 

\begin{theorem}
\label{asymptotic for M}
\begin{equation}
\label{eq:asymptotic for M}
M_r(x;\bm{\omega},a)
=\dfrac{1}{\Gamma(x+1)}
\sum_{k=0}^r (-1)^k\Lambda_k(\bm{\omega}) \dfrac{(r-k)!}{x^{r-k}}+O(x).
\end{equation}
In particular,
\[
M_r(x;(1,1,\ldots,1),0)
=\dfrac{1}{\Gamma(x+1)} \dfrac{r!}{x^r}+O(x).
\]
\end{theorem}

For example,
\begin{align*}
M_1(x;\omega,a)
&=\dfrac{1}{x}+\gamma-\log \omega+O(x),\\
M_2(x;\bm{\omega},a)
&=\dfrac{2}{x^2}+\dfrac{2\gamma -\log (\omega_1\omega_2)}{x}
+(\gamma-\log \omega_1)(\gamma-\log \omega_2)-\dfrac{\pi^2}{6}+O(x).
\end{align*}

The behavior of Mordell-Tornheim multiple zeta-functions around singularities has not yet been well
studied.
When $r=2$, Komori \cite[Theorem 3]{Komo08} proved such a result by using 
a certain contour integral argument, and another approach was developed by
the second named author \cite[Corollary 4.1]{Onod14}.
But it seems that both of their results do not cover the case $r=2$ of our result.
The known special cases of our result appear to be only the asymptotic formulas
for $M_1(x;\omega,a)$ and $M_2(x;\bm{\omega},0)$,
which have been essentially proved by Dixit, Sathyanarayana, and Sharan very recently
(see Theorem 1.1 and its proof in \cite{DSS24}). 
Theorem \ref{asymptotic for M} gives a kind of Laurent type expansion of
$\zeta_{MT,r}(1,\ldots,1,x)$ around the point $x=0$.   It is desirable to study the
behavior of Mordell-Tornheim (general $r$-fold) multiple zeta-functions around 
other singularity points.

Here we mention
another important class of multiple series, that is the Euler-Zagier multiple zeta-functions
defined by
\begin{align}
\label{EZ-series}
\zeta_{EZ,r}(s_1,s_2,\ldots,s_r)=\sum_{1\leq n_1<n_2<\cdots < n_r}\frac{1}
{n_1^{s_1}n_2^{s_2}\cdots n_r^{s_r}},
\end{align}
where $s_1,\ldots,s_r$ are complex variables.   This type of series has been studied 
more extensively.    It can be continued meromorphically to the whole space 
$\mathbb{C}^r$, and the list of all singular hyperplanes is known
(see Akiyama et al. \cite{AET01}).   
Almost all integer points are on one of those hyperplanes.
The behavior of $\zeta_{EZ,r}(s_1,\ldots,s_r)$ around such integer points was first
already studied in \cite{AET01}, and then discussed by many subsequent papers.
In particular, Laurent type expansions of $\zeta_{EZ,r}(s_1,\ldots,s_r)$ around
integer points have been investigated by Matsumoto et al. \cite{MOW20} and by Saha 
\cite{Saha22}.
In addition, here is just one of many results on $\zeta_{EZ,r}(s_1,\ldots,s_r)$.
In \cite[Proposition 4 (ii)]{AK99}, Arakawa and Kaneko proved the asymptotic formula
\begin{equation}
\label{eq:asymptotic for MZF}
\zeta_{EZ,r}(\underbrace{1,\ldots,1}_{r-1},x+1)
=\dfrac{1}{\Gamma(x+1)}\dfrac{1}{x^r}+O(x).
\end{equation}
This can also be shown from Theorem \ref{asymptotic for M}
by investigating the relation between $M_r(x;(1,\ldots,1),0)$ and $\zeta_{EZ,r}(1,\ldots,1,x+1)$
(see Sec.\ \ref{sec:asymptotic for M}).

Thanks to Theorems \ref{asymptotic for I} and \ref{asymptotic for M}, the behavior of $I_r(x;\bm{\omega},a)$ and $M_r(x;\bm{\omega},a)$ around $x=0$ can be described up to the constant term.
As a further study, it is natural to consider their complete asymptotic expansions.
In this paper, we address this problem for $I_r(x;\bm{\omega},a)$ using two methods.
To state the details, we introduce some notation.
Let $\Zb_{\ge 1}$ (resp.\ $\Zb_{\ge 0}$) denote the set of the positive (resp.\ non-negative) integers and let $R=\{1,2,\ldots,r\}$.
For $J=\{j_1,j_2,\ldots,j_s\}\subset R$ with $j_1<j_2<\cdots <j_s$, we set
\[
\bm{\omega}_J=(\omega_{j_1},\omega_{j_2},\ldots,\omega_{j_s})
\qquad
|\bm{\omega}_J|=\omega_{j_1}+\omega_{j_2}+\cdots +\omega_{j_s}
=\sum_{j\in J} \omega_j,
\]
where $\bm{\omega}_\emptyset$ and $|\bm{\omega}_\emptyset|$ are interpreted as $\emptyset$ and $0$, respectively.
Also, put $|\bm{\omega}|=|\bm{\omega}_R|$ for simplicity.
Further, we define  the numbers $a_{k,m}$ ($k,m=0,1,2,\ldots$) by
\begin{align}
	\label{v9 a(k,m) sequence}
	\dfrac{e^{-\gamma x}}{\Gamma(x+1)} \biggl(\dfrac{d}{d x} \biggr)^{k}
	\dfrac{\Gamma(x+1)}{e^{-\gamma x}}
	=\sum_{m=0}^\infty a_{k,m} x^m.
\end{align}
The above Taylor series is valid for $ 0<x<1 $.

Our first method refines the argument used in the proof of Theorem \ref{asymptotic for I},
and thereby yields the following more detailed result.

\begin{theorem}
	\label{v9 first complete asymptotic expansion of integral}
	For $ 0<x<1 $ and $ |\bm{\omega}|<a $, we have
	\begin{align*}
		a^x I_r(x;\bm{\omega},a)
		=\sum_{m=0}^\infty d_{r,m}(\bm{\omega},a) x^{m-r},
	\end{align*}
where  for $0\le m\le r$, 
\[
d_{r,m}(\bm{\omega},a)
=\sum_{j=0}^{m}
(r-j)!
\Lambda_j(\bm{\omega}/a)
\sum_{k=0}^{m-j}
(-1)^{j+k}\dfrac{a_{k,m-j-k}}{k!}
\]
and for $m>r$, 
\begin{align*}
	d_{r,m}(\bm{\omega},a)
	&=\sum_{j=0}^{r}
	(r-j)!
	\Lambda_j(\bm{\omega}/a)
	\sum_{k=0}^{r-j}
	(-1)^{j+k}\dfrac{a_{k,m-j-k}}{k!}\\
	& \quad +
	(-1)^r \sum_{\substack{A,B,C \\ C\ne \emptyset}}
	\biggl(\prod_{i\in A} \log \dfrac{\omega_i}{a}\biggr)
	\sum_{j=1}^{m-r+|B|}
	T_{|C|,j}\biggl(-\dfrac{\bm{\omega}_C}{a}\biggr)\\
	&\quad \times
	\sum_{k=\max\{0,j-m+r\}}^{\min\{|B|,j-1\}}
	\binom{|B|}{k} (j-k)_k a_{|B|-k,m-r-j+k}\,,
\end{align*}
where $A$, $B$, and $C$ are disjoint subsets of $R$ satisfying $A\sqcup B\sqcup C=R$, and $ T_{r,l}(\bm{\omega}) $ is defined in Proposition \ref{power series of S}.
	\end{theorem}
In Section \ref{sec:refinement}, we study these numbers $  a_{k,m} $ in detail and also compute the coefficients $ d_{r,m}(\bm{\omega},a) $ for particular instances (see Example \ref{v9 example for coefficients d}).
As an illustration, the following  complete asymptotic expansion of $a^x I_1(x;\omega,a)$ holds: 
If $0<\omega<a$ and $0<x<1$, then
\begin{equation}
\label{eq:I1-2}
a^x I_1(x;\omega,a)
=\dfrac1x+\log \dfrac{a}{\omega}
-\sum_{k=1}^\infty \Bigl\{
\Li _{\underbrace{\mbox{\scriptsize $1,\ldots,1$}}_{k-1},2} \Bigl(-\dfrac{\omega}{a}\Bigr)
+(-1)^{k+1}\zeta(k+1)
\Bigr\} x^k,
\end{equation}
where $\Li_{k_1,k_2,\ldots,k_l}(z)$ is the one-variable multiple polylogarithm defined as
\begin{equation}
\label{eq:def of one-variable MPL}
\Li_{k_1,k_2,\ldots,k_l}(z)
=\Li_{k_1,k_2,\ldots,k_l}(\underbrace{1,\ldots,1}_{l-1},z)
=\sum_{1\le n_1<n_2<\cdots <n_l} \dfrac{z^{n_l}}{n_1^{k_1}n_2^{k_2}\cdots n_l^{k_l}}
\end{equation}
for $z\in \Cb$ with $|z|<1$.
\par
Our second method of searching for complete asymptotic expansions is entirely different, using the multiple polylogarithm defined by \eqref{def of MPL}.
The result can be stated as follows.

\begin{theorem}
\label{power series for I}
If $0<x<1$, then
\begin{equation}
\label{eq:power series for I}
(a+|\bm{\omega}|)^x I_r(x;\bm{\omega},a)
=r! x^{-r}+\sum_{m=1}^\infty c_{r,m}(\bm{\omega},a) x^{m-r},
\end{equation}
where
\begin{align}
\label{eq:def of c}
&c_{r,m}(\bm{\omega},a)\\
&=\sum_{t=1}^{\min\{m,r\}}
(-1)^{m-t}
\sum_{s=1}^t
\sum_{\substack{\bm{k}\in (\Zb_{\ge 1})^s \\ k_1+\cdots+k_s=t}}
\sum_{\substack{\bm{l}\in (\Zb_{\ge 0})^s \\ l_1+\cdots+l_s=m-t}}
(r-t)!k_1!\cdots k_s!
\prod_{i=1}^s \binom{k_i+l_i-1}{l_i} \notag\\
&\quad
\times \sum_{K_1,\ldots,K_s}
\Li_{\bm{k}+\bm{l}}
\biggl(
\dfrac{a+|\bm{\omega}_{K_0}|}{a+|\bm{\omega}_{K_0\sqcup K_1}|},
\dfrac{a+|\bm{\omega}_{K_0\sqcup K_1}|}{a+|\bm{\omega}_{K_0\sqcup K_1\sqcup K_2}|},
\ldots,\dfrac{a+|\bm{\omega}_{K_0\sqcup \cdots \sqcup K_{s-1}}|}{a+|\bm{\omega}|}
\biggr) \notag
\end{align}
and $K_1,\ldots,K_s$ are disjoint subsets of $R$
satisfying $|K_i|=k_i$ $(i=1,\ldots,s)$
and  $K_0=R\setminus \bigsqcup_{i=1}^s K_i$.
\end{theorem}

For example, we have
\begin{equation}
\label{eq:I1-1}
(a+\omega)^x I_1(x;\omega,a)
=\dfrac1x+\sum_{k=0}^\infty (-1)^k \Li_{k+1}\Bigl(\dfrac{a}{a+\omega}\Bigr)
x^k
\end{equation}
and
\begin{align*}
(a+\omega_1+\omega_2)^x I_2(x;\bm{\omega},a)
&= \dfrac{2}{x^2}
+\Bigl\{\Li_1 \Bigl(\dfrac{a+\omega_1}{a+\omega_1+\omega_2}\Bigr)
+\Li_1 \Bigl(\dfrac{a+\omega_2}{a+\omega_1+\omega_2}\Bigr)\Bigr\}
\dfrac1x\\
&\quad +\sum_{m=2}^\infty c_{2,m}(\bm{\omega},a) x^{m-2},
\end{align*}
where
\begin{align*}
c_{2,m}(\bm{\omega},a)
&=(-1)^{m-1}\sum_{j=1,2} \Li_m\biggl(\dfrac{a+\omega_j}{a+\omega_1+\omega_2}\biggr)
+(-1)^m 2(m-1) \Li_m\biggl(\dfrac{a}{a+\omega_1+\omega_2}\biggr)\\
&\quad
+(-1)^m \sum_{l=1}^{m-1}
\sum_{j=1,2} \Li_{l,m-l} \biggl(\dfrac{a}{a+\omega_j},\dfrac{a+\omega_j}{a+\omega_1+\omega_2}\biggr).
\end{align*}

Our results concerning the behavior of $I_r(x;\bm{\omega},a)$ have the following
interesting application to multiple polylogarithms.
Theorems \ref{asymptotic for I} and \ref{power series for I} give two asymptotic formulas of $(a+|\bm{\omega}|)^x I_r(x;\bm{\omega},a)$ which are apparently different.
Hence, by comparing the coefficients of $x^{m-r}$ for $m\ge 1$,
we obtain some nontrivial relations among multiple polylogarithms.
For instance, when $r=2$ and $m=2$, we have
\begin{align}
\label{eq:example for r=m=2}
&-\Li_2 \Bigl(\dfrac{a+\omega_1}{a+\omega_1+\omega_2}\Bigr)
-\Li_2 \Bigl(\dfrac{a+\omega_2}{a+\omega_1+\omega_2}\Bigr)
+2\Li_2 \Bigl(\dfrac{a}{a+\omega_1+\omega_2}\Bigr)\\
&
+\Li_{1,1}\Bigl(\dfrac{a}{a+\omega_1},\dfrac{a+\omega_1}{a+\omega_1+\omega_2}\Bigr)
+\Li_{1,1}\Bigl(\dfrac{a}{a+\omega_2},\dfrac{a+\omega_2}{a+\omega_1+\omega_2}\Bigr) \notag\\
&=\log \dfrac{\omega_1}{a+\omega_1+\omega_2} \cdot \log \dfrac{\omega_2}{a+\omega_1+\omega_2} -\dfrac{\pi^2}{6}. \notag
\end{align}
Further examples of relations among multiple polylogarithms will be presented
in Corollary \ref{cor:1,...,1,2} below and in Section \ref{sec:application to MPL}.

Those examples can be shown as special cases of
the following Theorem \ref{relation among MPL}, which express 
$c_{r,m}(\bm{\omega},a)$
as a finite sum of products of logarithms and values of Bell polynomials
$B_k(x_1,\ldots,x_k)$ (whose
definition will be given in Section \ref{sec:refinement}).

\begin{theorem}
\label{relation among MPL}
Let $c_{r,m}(\bm{\omega},a)$ be defined as in \eqref{eq:def of c}.
If $1\le m \le r$, then
\begin{align}
\label{eq:relation among MPL}
c_{r,m}(\bm{\omega},a)
&=
\sum_{k=0}^m (-1)^{m-k}\dfrac{(r-m+k)!}{k!}
\Lambda_{m-k}\left(\dfrac{\bm{\omega}}{a+|\bm{\omega}|}\right)\\
&\quad \times
B_k(0,-1!\zeta(2),2!\zeta(3),\ldots,(-1)^{k-1}(k-1)!\zeta(k)).
\notag
\end{align}
\end{theorem}

Here, the values of Bell polynomials appear as polynomials in Riemann zeta values with integer coefficients.
Since $c_{r,m}(\bm{\omega},a)$ is a linear combination of multiple polylogarithms with integer coefficients by \eqref{eq:def of c}, the above theorem gives nontrivial relations
among multiple polylogarithms.
\par
Theorems \refeq{v9 first complete asymptotic expansion of integral} and \ref{power series for I} provide the complete asymptotic expansions of the integral $I_r(x;\bm{\omega},a)$.  Consequently, they establish the following specific relationships among the expansion coefficients $c_{r,m}(\bm{\omega},a),d_{r,m}(\bm{\omega},a)$. Alternatively, these results yield nontrivial identities involving multiple polylogarithms.
\begin{theorem}
	\label{v9 relation between coefficients c and d}
	For  integer $ m\geq 0 $ and  $ |\bm{\omega}|<a $, we have\\
	\begin{align}
		\label{v9 d in terms c}
		d_{r,m}(\bm{\omega},a)
	=\dfrac{r!}{m!} \biggl(\log \dfrac{a}{a+|\bm{\omega}|}\biggr)^m
	+\sum_{k=1}^m \biggl(\log \dfrac{a}{a+|\bm{\omega}|}\biggr)^{m-k}
	\dfrac{c_{r,k}(\bm{\omega},a)}{(m-k)!}
\end{align}
and 
	\begin{align}
		\label{v9 c in terms d}
		c_{r,m}(\bm{\omega},a)
	=\dfrac{r!}{m!} \biggl(\log \dfrac{a+|\bm{\omega}|}{a}\biggr)^m
	+\sum_{k=1}^m \biggl(\log \dfrac{a+|\bm{\omega}|}{a}\biggr)^{m-k}
	\dfrac{d_{r,k}(\bm{\omega},a)}{(m-k)!}.
\end{align}
	\end{theorem}

For instance, by taking $ r=1 $ and $ \frac{\omega}{a}=x $  in Theorem \ref{v9 relation between coefficients c and d}, we obtain the following relations, which imply that each $\Li_{1,\ldots,1,2}(x)$ can be expressed in terms of polylogarithms, and vice versa.
\begin{corollary}
	\label{cor:1,...,1,2}
	For $ 0<x<1 $ and any positive integer $ k $, we have
	\begin{align}\label{cor-1-1}
		&\Li_{\underbrace{\mbox{\scriptsize $1,\ldots,1$}}_{k-1},2}\left(-x\right)
		+(-1)^{k+1}\zeta(k+1)\\
		&\hspace{4mm}=-\dfrac{1}{(k+1)!} \biggl(\log \dfrac{1}{x+1}\biggr)^{k+1}
		+\sum_{j=0}^k \dfrac{(-1)^{j+1}}{(k-j)!}
		\biggl(\log \dfrac{1}{x+1}\biggr)^{k-j}
		\Li_{j+1}\biggl(\dfrac{1}{x+1}\biggr)\notag
	\end{align}
	and
	\begin{align}\label{cor-1-2}
		&\Li_{k+1}\biggl(\dfrac{1}{x+1}\biggr)\\
		&\hspace{4mm}=
		-\dfrac{1}{(k+1)!} \biggl(\log \dfrac{1}{x+1}\biggr)^{k+1}
		+\dfrac{1}{k!} \log \dfrac{1}{x} \cdot \biggl(\log \dfrac{1}{x+1}\biggr)^k\notag\\
		&\hspace{9mm}
		+\sum_{j=1}^k \dfrac{(-1)^{j+1}}{(k-j)!}
		\biggl(\log \dfrac{1}{x+1}\biggr)^{k-j}
		\biggl(\Li_{\underbrace{\mbox{\scriptsize $1,\ldots,1$}}_{j-1},2}
		\left(-x \right)
		+(-1)^{j+1} \zeta(j+1)\biggr).\notag
	\end{align}
\end{corollary}

We will prove Theorems \ref{asymptotic for I}, \ref{asymptotic for M}, \ref{v9 first complete asymptotic expansion of integral} and \ref{power series for I}
 in Sections \ref{sec:asymptotic for I}, \ref{sec:asymptotic for M}, 
\ref{sec:refinement}, and \ref{sec:power series for I}, respectively. 
Section \ref{sec:asymptotic for M} also contains the proof of \eqref{eq:asymptotic for MZF} as an application of Theorem \ref{asymptotic for M}.  
 In Section \ref{sec:application to MPL}, we will present the proofs of Theorems \ref{relation among MPL}, \ref{v9 relation between coefficients c and d}
and also provide some examples.


\section{Proof of Theorem \ref{asymptotic for I}}
\label{sec:asymptotic for I}
Before the proof of Theorem \ref{asymptotic for I}, we first notice:

\begin{proposition}
\label{convergence for I}
The integral \eqref{def of I} is convergent if and only if $x>0$.
\end{proposition}

\begin{proof}
For $x>0$, since
\begin{equation}
\label{eq:AGM}
\omega_1t_1+\cdots+\omega_rt_r+a
\geq \omega_1t_1+\cdots+\omega_rt_r\\
\geq r (\omega_1\cdots \omega_r)^{\frac{1}{r}}
(t_1\cdots t_r)^{\frac{1}{r}},
\end{equation}
we have
\[
I_r(x;\bm{\omega},a)
\ll  \int_1^\infty \cdots \int_1^\infty
\dfrac{dt_1\cdots dt_r}{t_1\cdots t_r (t_1\cdots t_r)^{\frac{x}{r}}}
=\prod_{i=1}^{r}\left(\int_1^\infty \frac{dt_i}{t_i^{1+\frac{x}{r}}}\right)<\infty\,.
\]
The converse part follows from the fact that, for $x\leq 0$, 
\begin{align*}
I_r(x;\bm{\omega},a)
&=\int_1^\infty \cdots \int_1^\infty
\frac{(\omega_1t_1+\cdots+\omega_r t_r+a)^{-x}}{t_1\cdots t_r}\, dt_1\cdots dt_r\\
&\ge
\int_1^\infty \cdots \int_1^\infty
\frac{(|\bm{\omega}|+a)^{-x}}{t_1\cdots t_r}\, dt_1\cdots dt_r
=(|\bm{\omega}|+a)^{-x}\prod_{i=1}^{r}\left(\int_1^\infty \frac{1}{t_i}\, dt_i\right),
\end{align*}
which is divergent.
	\end{proof}

To prove Theorem \ref{asymptotic for I}
and Theorem \ref{expression by S} below,
we need the following asymptotic behavior of  the incomplete gamma function
\[
\Gamma(s,u)=\int_u^\infty e^{-t} t^{s-1} \, dt,
\]
where $s$ is a complex number and $u>0$.

\begin{lemma}
For $u>0$, we have
\begin{equation}
\label{incomplete gamma}
\Gamma(0,u)=-\log u-\gamma-\sum_{n=1}^\infty \dfrac{(-u)^n}{n\cdot n!}.
\end{equation}
\end{lemma}

\begin{proof}
If  $s>0$, then we can write $\Gamma(s,u)$ as
\[
\Gamma(s,u)
=\Gamma(s)-\int_0^u e^{-t}t^{s-1}\, dt
\]
and expand the integral into a power series in $u$ as follows:
\[
\int_0^u e^{-t}t^{s-1}\, dt
=\int_0^u \biggl(\sum_{n=0}^\infty \dfrac{(-t)^n}{n!}\biggr) t^{s-1}\, dt
=\sum_{n=0}^\infty \dfrac{(-1)^n u^{n+s}}{(n+s)\cdot n!}.
\]
Hence, we have
\[
\Gamma(0,u)=\lim_{s\to 0^+} \biggl(\Gamma(s)-\dfrac{u^s}{s}\biggr)
-\sum_{n=1}^\infty \dfrac{(-u)^n}{n\cdot n!}.
\]
Since the limit coincides with $\Gamma'(1)-\log u=-\gamma-\log u$,
we obtain \eqref{incomplete gamma}.
\end{proof}

We now prove Theorem \ref{asymptotic for I}.
By using the integral expression of $\Gamma(x)$, we see
\begin{align}
\label{expression by incomplete Gamma}
I_r(x;\bm{\omega},a)
&=
\int_1^\infty \cdots \int_1^\infty
\dfrac{1}{t_1\cdots t_r}
\cdot 
\dfrac{1}{\Gamma(x)}
\int_0^\infty
e^{-(\omega_1t_1+\cdots +\omega_rt_r+a)u}u^{x-1}\, du dt_1\cdots dt_r\\[5pt]
&=\dfrac{1}{\Gamma(x)}
\int_0^\infty
\biggl(\int_1^\infty \dfrac{e^{-\omega_1 t_1u}}{t_1}\, dt_1\biggr)
\cdots 
\biggl(\int_1^\infty \dfrac{e^{-\omega_r t_ru}}{t_r}\, dt_r\biggr)
\cdot e^{-au} u^{x-1}\, du \notag\\[5pt]
&=\dfrac{1}{\Gamma(x)}
\int_0^\infty \prod_{i=1}^r \Gamma(0,\omega_i u)
\cdot e^{-au} u^{x-1}\, du. \notag
\end{align}
We apply the inequality $\Gamma(0,u)\ll e^{-u/2}$ as $u\to\infty$
to obtain 
\[
I_r(x;\bm{\omega},a)
=\dfrac{1}{\Gamma(x)}
\int_0^1 \prod_{i=1}^r \Gamma(0,\omega_i u)
\cdot e^{-au} u^{x-1}\, du+O(x).
\]
Note that
$\Gamma(0,\omega_iu)=-\log \omega_i-\log u-\gamma +O(u)$
by \eqref{incomplete gamma}
and that $e^{-au}=1+O(u)$.
Hence we get
\begin{align*}
I_r(x;\bm{\omega},a)
&=\dfrac{(-1)^r}{\Gamma(x)}
\int_0^1 \prod_{i=1}^r (\log \omega_i+\log u+\gamma) \cdot u^{x-1}\, du
+O(x)\\
&=\dfrac{(-1)^r}{\Gamma(x)}
\int_0^1 \biggl(\sum_{k=0}^r \Lambda_k(\bm{\omega}) (\log u+\gamma)^{r-k} \biggr) u^{x-1}\, du
+O(x)\\
&=\dfrac{(-1)^r}{\Gamma(x)}
\sum_{k=0}^r \Lambda_k(\bm{\omega}) \int_0^1 (\log u+\gamma)^{r-k} u^{x-1}\, du
+O(x).
\end{align*}
As for the integral in the above, we see that
\begin{align*}
\int_0^1 (\log u+\gamma)^{r-k} u^{x-1}\, du
&=e^{-\gamma x} \int_0^1 \biggl(\dfrac{\partial}{\partial x}\biggr)^{r-k}
e^{\gamma x} u^{x-1}\, du
=
e^{-\gamma x} \biggl(\dfrac{\partial}{\partial x}\biggr)^{r-k} \dfrac{e^{\gamma x}}{x}\\
&= e^{-\gamma x} \cdot \dfrac{(-1)^{r-k} (r-k)!}{x^{r-k+1}}+O(1),
\end{align*}
where the last equality can be shown by considering the power series expansion of 
$e^{\gamma x}$ and differentiating termwisely. 
Thus we obtain Theorem \ref{asymptotic for I}.

\section{Proof of Theorem \ref{asymptotic for M}}
\label{sec:asymptotic for M}
Before proving Theorem \ref{asymptotic for M}, we remark the following proposition,
which can be easily derived using a method similar to the proof of Proposition \ref{convergence for I}.

\begin{proposition}
\label{convergence}
The series \eqref{function M} is convergent if and only if $x>0$.
\end{proposition}

In this section, we present two proofs of Theorem \ref{asymptotic for M}.
In the first proof, we apply Theorem  \ref{asymptotic for I},
and so it is essential to reduce the estimate of $M_r(x;\bm{\omega},a)$ to that of $I_r(x;\bm{\omega},a)$.
This reduction can be achieved using Proposition \ref{relation between M and I} below.
We begin with preparing the following several lemmas necessary for the proof of Proposition \ref{relation between M and I}.

\begin{lemma}
\label{Abel}
The multiple harmonic series \eqref{function M} can be expressed in  terms of multiple integrals in the form
\begin{align}
\label{eq:Abel}
M_r(x;\bm{\omega},a)
&=\omega_1\cdots \omega_r (x)_r
\int_1^\infty \cdots \int_1^\infty
\biggl(\sum_{n_1\le t_1} \dfrac{1}{n_1}\biggr)
\cdots
\biggl(\sum_{n_r\le t_r} \dfrac{1}{n_r}\biggr)\\
&\hskip35mm \times 
\dfrac{1}{(\omega_1t_1+\cdots+\omega_rt_r+a)^{x+r}}
dt_1\cdots dt_r, \notag
\end{align}
where $ (x)_r $ is the Pochhammer symbol, defined by, $ (x)_r=x(x+1)\cdots (x+r-1) $.
\end{lemma}

\begin{proof}
	First we have to check the existence of the multiple integral on the
	right-hand side of \eqref{eq:Abel}.
	As $ \displaystyle \sum_{n\leq t}n^{-1}\ll t^{\frac{x}{2r}} $ and \eqref{eq:AGM},
	we have
	\begin{align*}
	&\int_{1}^{\infty}\cdots \int_{1}^{\infty} \left(\sum_{n_1\leq t_1}\frac{1}{n_1}\right)	\cdots \left(\sum_{n_r\leq t_r}\frac{1}{n_r}\right)\frac{1}{(\omega_1t_1+\cdots+\omega_rt_r+a)^{x+r}}dt_1\cdots dt_r\\
	&\ll \int_{1}^{\infty}\cdots \int_{1}^{\infty}\frac{t_1^{\frac{x}{2r}}\cdots t_r^{\frac{x}{2r}}}{ (t_1\cdots t_r)^{1+\frac{x}{r}}}dt_1 \cdots dt_r=\left(\int_{1}^{\infty}\frac{1}{t^{1+\frac{x}{2r}}}dt\right)^r=\left(\frac{2r}{x}\right)^r < \infty\,.
	\end{align*}
	
Let $N_1,\ldots,N_r$ be positive integers, and we start with
	\begin{align}
		\label{Lemma1,exp1}
	&\sum_{n_i\leq N_i,\,1\leq i \leq r}\frac{1}{n_1n_2\cdots n_r(\omega_1n_1+\cdots+\omega_rn_r+a)^x}\\&\qquad=\sum_{n_i\leq N_i,\,1\leq i \leq r-1}\frac{1}{n_1n_2\cdots n_{r-1}}\sum_{n_r\leq N_r}\frac{1}{n_r(\omega_1n_1+\cdots+\omega_rn_r+a)^x}.\notag
	\end{align}
Applying Abel's summation formula, we find that the inner sum 
is equal to
\begin{align}
	\label{lemma1,exp2}
&\frac{\sum_{n_r\leq N_r}\frac{1}{n_r}}{(\omega_1 n_1+\cdots+\omega_{r-1}n_{r-1}+\omega_r N_r+a)^x}\\
&\quad
+\omega_rx\int_{1}^{N_r}\left(\sum_{n_r\leq t_r}\frac{1}{n_r}\right)\frac{1}{(\omega_1n_1+\cdots+\omega_{r-1}n_{r-1}+\omega_rt_r+a)^{x+1}}dt_r\,. \notag
\end{align}
Since $ \sum_{n\leq N}n^{-1}=O\left(N^\epsilon\right) $ for any $ \epsilon>0 $, we get from \eqref{Lemma1,exp1} and \eqref{lemma1,exp2} by taking $ N_r \to\infty $ that
\begin{align*}
	&\sum_{\substack{n_i\leq N_i,\,1\leq i \leq r-1\\n_r\geq 1}}	\frac{1}{n_1n_2\cdots n_r(\omega_1n_1+\cdots+\omega_rn_r+a)^x}\\
	&=\omega_rx\sum_{n_i\leq N_i,\,1\leq i \leq r-1}\frac{1}{n_1n_2\cdots n_{r-1}}\int_{1}^{\infty}\left(\sum_{n_r\leq t_r}\frac{1}{n_r}\right)\\
	&\qquad\qquad\times \frac{1}{(\omega_1n_1+\cdots+\omega_{r-1}n_{r-1}+\omega_rt_r+a)^{x+1}}dt_r\\
	&=\omega_rx\sum_{n_i\leq N_i,\,1\leq i \leq r-2}\frac{1}{n_1n_2\cdots n_{r-2}}\int_{1}^{\infty}\left(\sum_{n_r\leq t_r}\frac{1}{n_r}\right)\\
	&\qquad\qquad\times\sum_{n_{r-1}\leq N_{r-1}}\frac{1}{n_{r-1}(\omega_1n_1+\cdots+\omega_{r-1}n_{r-1}+\omega_rt_r+a)^{x+1}}dt_r\,.
\end{align*}
Again applying the Abel summation formula to 
the innermost sum similarly and then taking $ N_{r-1} $ tends to $ \infty $, we get that
\begin{align*}
&\sum_{\substack{n_i\leq N_i,\,1\leq i \leq r-2\\n_{r-1},n_r\geq 1}}	\frac{1}{n_1n_2\cdots n_r(\omega_1n_1+\cdots+\omega_rn_r+a)^x}\\
&=\omega_1\omega_2(x)_2	\sum_{n_i\leq N_i,\,1\leq i \leq r-2}\frac{1}{n_1n_2\cdots n_{r-2}}\int_{1}^{\infty}\int_{1}^{\infty}\left(\sum_{n_{r-1}\leq t_{r-1}}
\frac{1}{n_{r-1}}\right)\left(\sum_{n_r\leq t_r}\frac{1}{n_r}\right)\\
&\qquad\qquad\times\frac{1}{(\omega_1n_1+\cdots+\omega_{r-2}n_{r-2}+\omega_{r-1}t_{r-1}+\omega_r t_r+a)^{x+2}}dt_{r-1}dt_r\,.
\end{align*}
After continuing this process $ r $-times, we reach at the multiple integral \eqref{eq:Abel}.
	\end{proof}

\begin{lemma}
\label{integration by parts}
For $J\subset R$, we have
\begin{equation}
\label{eq:integration by parts}
\omega_1\cdots \omega_r (x)_r
\int_1^\infty \cdots \int_1^\infty
\dfrac{(\prod_{j\in J} \log t_j )\,dt_1\cdots dt_r}{(\omega_1t_1+\cdots+\omega_r t_r+a)^{x+r}}
=I_{|J|}(x;\bm{\omega}_J,a+|\bm{\omega}_{R\setminus J}|),
\end{equation}
where $I_0(x;\emptyset,a)$ is interpreted as $a^{-x}$.
\end{lemma}

\begin{proof}
Let $L$ denote the left-hand side of \eqref{eq:integration by parts}.
In  the case $J=\emptyset$,
we can easily see $L=(a+|\bm{\omega}|)^{-x}$,
which coincides with $I_0(x;\bm{\omega}_{\emptyset},a+|\bm{\omega}_R|)$.
Thus \eqref{eq:integration by parts} holds in this case.

We next consider the case $J\neq \emptyset$.
By applying integration by parts, one can get that for any fixed $k\in J$,
\begin{align*}
&\int_{1}^{\infty} \frac{ \log t_k}{(\omega_1t_1+\cdots +\omega_rt_r+a)^{x+r}}dt_{k}\\
&=
\frac{1}{\omega_k(x+r-1)}\int_{1}^{\infty}\frac{1}{t_k(\omega_1t_1+\cdots+\omega_rt_r+a)^{x+r-1}}dt_k.	
\end{align*}
Hence we have
\begin{align*}
L&=\biggl(\prod_{j\in R\setminus\{k\}} \omega_j\biggr) \cdot 
(x)_{r-1}\int_{1}^{\infty}\cdots \int_{1}^{\infty}
\frac{\prod_{j\in J\setminus \{k\}} \log t_j}{t_k(\omega_1t_1+\cdots+\omega_rt_r+a)^{x+r-1}}dt_1\cdots dt_r.
\end{align*}
After continuing this process $|J|$-times, we get that
\begin{align*}
L&=\int_{1}^{\infty}\cdots \int_{1}^{\infty} \frac{(\prod_{i\in R\setminus J} \omega_i)\cdot (x)_{r-|J|}}{(\prod_{j\in J} t_j)(\omega_1t_1+\cdots+\omega_rt_r+a)^{x+r-|J|}}dt_1 \cdots dt_r\\
&=\int_{1}^{\infty}\cdots \int_{1}^{\infty}\left(\int_{1}^{\infty}\cdots \int_{1}^{\infty}
\frac{(\prod_{i\in R\setminus J} \omega_i)\cdot (x)_{r-|J|}}{(\omega_1t_1+\cdots+\omega_rt_r+a)^{x+r-|J|}}\prod_{i\in R\setminus J} dt_i\right)
\prod_{j\in J} \dfrac{dt_j}{t_j},
\end{align*}
where the inner multiple integral should be interpreted as $(\omega_1t_1+\cdots +\omega_rt_r+a)^{-x}$ if $J=R$.
One can see that
the inner multiple integral equals
$(\sum_{j\in J} \omega_j t_j+a+|\bm{\omega}_{R\setminus J}|)^{-x}$
and hence \eqref{eq:integration by parts} follows.
\end{proof}

\begin{lemma}
\label{estimate for I}
If $r$ is a positive integer, then
\[
I_r(x;\bm{\omega},a)
\le \dfrac{r^{r-x}}{(\omega_1\cdots \omega_r)^{\frac{x}{r}} x^r}.
\]
\end{lemma}

\begin{proof}
By using \eqref{eq:AGM}, we have
\[
I_r(x;\bm{\omega},a)
\leq \dfrac{1}{r^x (\omega_1\cdots \omega_r)^{\frac{x}{r}}}
\int_1^\infty \cdots \int_1^\infty
\dfrac{1}{t_1\cdots t_r (t_1\cdots t_r)^{\frac{x}{r}}}dt_1\cdots dt_r.
\]
The multiple integral in the right-hand side is equal to
\[
\left(\int_1^\infty \dfrac{1}{t^{1+\frac{x}{r}}} dt\right)^r
=\left(\dfrac{r}{x}\right)^r.
\]
Hence we obtain Lemma \ref{estimate for I}.
\end{proof}

\begin{proposition}
\label{relation between M and I}
The multiple harmonic series \eqref{function M} can be expressed in  terms of its integral analogues in the form
\[
M_r(x;\bm{\omega},a)
=\sum_{k=0}^r 
\gamma^k
\sum_{\substack{J\subset R \\ |J|=r-k}}
I_{r-k}(x;\bm{\omega}_J, a+|\bm{\omega}_{R\setminus J}|)
+O(x).
\]
\end{proposition}

\begin{proof}
We begin with the expression \eqref{eq:Abel} of $M_r(x;\bm{\omega},a)$ given by 
Lemma \ref{Abel}.   On the right-hand side of \eqref{eq:Abel},
there are factors
$$
\left(\sum_{n_1\leq t_1}\frac{1}{n_1}\right)	\cdots \left(\sum_{n_r\leq t_r}
	\frac{1}{n_r}\right).
$$
We apply $\sum_{n\le t} n^{-1}=\log t +\gamma+O(t^{-1})$ to each factor of the above.
Then we encounter the
product 
$$ \left(\log t_1+\gamma\right) \left(\log t_2+\gamma\right)\cdots  \left(\log t_r+\gamma\right), $$ 
which can be expanded in the form
\[
\sum_{k=0}^r \gamma^k  \sum_{\substack{J\subset R\\ |J|=r-k}} \prod_{j\in J} \log t_j.
\]
Therefore from Lemma \ref{integration by parts}, it follows that
\begin{align*}
M_r(x;\bm{\omega},a)
&=\sum_{k=0}^r \gamma^k \sum_{\substack{J\subset R\\ |J|=r-k}}
I_{r-k}(x;\bm{\omega}_J,a+|\bm{\omega}_{R\setminus J}|)\\
&\quad
+\omega_1\cdots \omega_r(x)_r
\int_{1}^{\infty} \cdots \int_{1}^{\infty}\frac{g(t_1,\ldots,t_r)}{(\omega_1t_1+\cdots +\omega_rt_r+a)^{x+r}}dt_1\cdots dt_r\,,\notag
\end{align*}
where $g(t_1,\ldots,t_r)$ is a quantity which satisfies the estimate
\begin{equation}
\label{g_est}
g(t_1,\ldots,t_r)
=O\left(\sum_{J\subsetneq R}
\sum_{i \in R\setminus J} \dfrac{1}{t_i} \prod_{j\in J} \log t_j
\right).
\end{equation}
Now, in order
to complete the proof of Proposition \ref{relation between M and I}, the only remaining task is 
to show that
\begin{align}
	\label{theorem1,exp1}
\omega_1\cdots \omega_r (x)_r\int_{1}^{\infty} \cdots \int_{1}^{\infty}\frac{g(t_1,\ldots,t_r)}{(\omega_1t_1+\cdots +\omega_rt_r+a)^{x+r}}dt_1\cdots dt_r=O(x)\,.
\end{align}
To justify \eqref{theorem1,exp1}, it is enough to verify that
\begin{equation}
	\label{theorem1,exp3}
\omega_1\cdots \omega_r(x)_r
\int_{1}^{\infty} \cdots \int_{1}^{\infty}
\frac{\prod_{j\in J} \log t_j}{t_i(\omega_1t_1+\cdots +\omega_r t_r+a)^{x+r}}dt_1\cdots dt_r=O(x)
\end{equation}
for $J\subsetneq R$ and $i\in R\setminus J$.
In the same way as the proof of Lemma \ref{integration by parts},
we find that the left-hand side of \eqref{theorem1,exp3} is equal to
 $\omega_i x I_{|J|+1}(x+1;\bm{\omega}_{J\sqcup \{i\}},a+|\bm{\omega}_{R\setminus(J\sqcup\{i\})}|)$.
Hence, by Lemma \ref{estimate for I}, we see
\[
\omega_i x I_{|J|+1}(x+1;\bm{\omega}_{J\sqcup \{i\}},a+|\bm{\omega}_{R\setminus(J\sqcup\{i\})}|)
\le
\dfrac{(|J|+1)^{|J|-x}\omega_i^{\frac{|J|-x}{|J|+1}} x}{(\prod_{j\in J} \omega_j)^{\frac{x+1}{|J|+1}}  (x+1)^{|J|+1}}
=O(x).
\]
This completes the proof of \eqref{theorem1,exp3}, and hence 
the proof of Proposition \ref{relation between M and I}.
\end{proof}

We now show Theorem \ref{asymptotic for M}.
We first remark that Theorem \ref{asymptotic for I} is true for $r=0$
if we interpret $I_0(x;\emptyset, a)$ as $a^{-x}$ (as in the statement of
Lemma \ref{integration by parts}).
Applying Proposition \ref{relation between M and I} and Theorem \ref{asymptotic for I} with this remark, we obtain
\begin{align*}
M_r(x;\bm{\omega},a)
&=
\sum_{k=0}^r \gamma^k \sum_{\substack{J\subset R\\ |J|=r-k}}
\dfrac{e^{-\gamma x}}{\Gamma(x+1)}
\sum_{l=0}^{r-k} (-1)^l \Lambda_l(\bm{\omega}_J)
\dfrac{(r-k-l)!}{x^{r-k-l}}+O(x)\\
&=\dfrac{e^{-\gamma x}}{\Gamma(x+1)}
\sum_{l=0}^r \sum_{k=0}^{r-l}
\dfrac{(-1)^l (r-k-l)! \gamma^k}{x^{r-k-l}}
\sum_{\substack{J\subset R\\ |J|=r-k}}
\Lambda_l(\bm{\omega}_J)+O(x).
\end{align*}
Since
\begin{align*}
\sum_{\substack{J\subset R\\ |J|=r-k}}
\Lambda_l(\bm{\omega}_J)
=\sum_{\substack{J\subset R\\ |J|=r-k}}
\sum_{\substack{I\subset J\\ |I|=l}}
\prod_{i\in I} \log \omega_i
=\sum_{\substack{I\subset R\\ |I|=l}}
\prod_{i\in I}\log \omega_i
\sum_{\substack{I\subset J\subset R\\ |J|=r-k}} 1
=\binom{r-l}{r-k-l} \Lambda_l(\bm{\omega}),
\end{align*}
we have
\[
M_r(x;\bm{\omega},a)
=\dfrac{e^{-\gamma x}}{\Gamma(x+1)}
\sum_{l=0}^r  (-1)^l \Lambda_l(\bm{\omega})
\dfrac{(r-l)!}{x^{r-l}}
\sum_{k=0}^{r-l} \dfrac{(\gamma x)^k}{k!}
+O(x).
\]
Note that
\[
\sum_{k=0}^{r-l} \dfrac{(\gamma x)^k}{k!}=e^{\gamma x}+O(x^{r-l+1}).
\]
This completes the proof of Theorem \ref{asymptotic for M}.

We next present the second proof of Theorem \ref{asymptotic for M},
which is similar to that of Theorem \ref{asymptotic for I}
and can be regarded as a generalization of the proof of Theorem 1.1 in \cite{DSS24}.
Using  the integral expression of $\Gamma(x)$, we have
\begin{align*}
M_r(x;\bm{\omega},a)
&=
\sum_{n_1,\cdots,n_r\ge 1} \dfrac{1}{n_1\cdots n_r}\cdot
\dfrac{1}{\Gamma(x)}
\int_0^\infty e^{-(\omega_1n_1+\cdots+\omega_rn_r+a)t} t^{x-1}\, dt\\[5pt]
&=\dfrac{1}{\Gamma(x)}
\int_0^\infty \left(\sum_{n_1=1}^\infty \dfrac{e^{-\omega_1n_1t}}{n_1}\right)
\cdots \left(\sum_{n_r=1}^\infty \dfrac{e^{-\omega_rn_rt}}{n_r}\right)
e^{-at} t^{x-1}\, dt\\[5pt]
&=\dfrac{(-1)^r}{\Gamma(x)}
\int_0^\infty \prod_{i=1}^r \log (1-e^{-\omega_it})
\cdot e^{-at}t^{x-1}\, dt.
\end{align*}
Since $\log (1-e^{-\omega_it})\ll e^{-\omega_it}$ as $t\to\infty$,
we see
\[
M_r(x;\bm{\omega},a)
=\dfrac{(-1)^r}{\Gamma(x)}
\int_0^1 \prod_{i=1}^r \log (1-e^{-\omega_it})
\cdot e^{-at}t^{x-1}\, dt+O(x).
\]
Moreover, we apply
$\log(1-e^{-\omega_i t})=\log \omega_i t+O(t)$
and $e^{-at}=1+O(t)$ as $t\to0^+$.
As a result, we obtain
\begin{align*}
M_r(x;\bm{\omega},a)
&=\dfrac{(-1)^r}{\Gamma(x)}
\int_0^1 \prod_{i=1}^r (\log \omega_i +\log t)\cdot t^{x-1}\, dt+O(x)\\
&=\dfrac{(-1)^r}{\Gamma(x)}
\int_0^1 \biggl(\sum_{k=0}^r \Lambda_k(\bm{\omega})
(\log t)^{r-k}\biggr) t^{x-1}\, dt+O(x)\\
&=\dfrac{(-1)^r}{\Gamma(x)}
\sum_{k=0}^r \Lambda_k(\bm{\omega})
\cdot (-1)^{r-k}\dfrac{ (r-k)!}{x^{r-k+1}}+O(x),
\end{align*}
which gives Theorem \ref{asymptotic for M}.

The second proof appears simpler than the first.
However, in the second proof, the involvement of the Euler constant
in the behavior of $M_r(x;\bm{\omega},a)$ around $x=0$
is cleverly obscured.
In contrast, the first proof reveals the contribution of the Euler constant  through Proposition \ref{relation between M and I}.
From this perspective, the first proof is also intriguing.
 
At the end of this section,
we prove the asymptotic formula \eqref{eq:asymptotic for MZF}
for the Euler-Zagier multiple zeta function
as an application of Theorem \ref{asymptotic for M}.
\begin{proposition}
\label{M-MZF}
For a positive real number $x$, we have
\begin{equation}
\label{eq:M-MZF}
M_r(x;(\underbrace{1,\ldots,1}_{r}),0)=
r!\zeta_{EZ,r}(\underbrace{1,\ldots,1}_{r-1},x+1)
\end{equation}
\end{proposition}

The relation \eqref{eq:M-MZF} for $x\in\Zb_{\ge 1}$
has been shown by Hoffman \cite[Corollary 4.2]{Hoff92}.

\begin{proof}[Proof of Proposition \ref{M-MZF}]
It follows from symmetry that
\begin{align*}
M_r(x;(\underbrace{1,\ldots,1}_{r}),0)
&=\sum_{n_1,\ldots,n_r\ge 1}
\dfrac{n_1+\cdots +n_r}{n_1\cdots n_r(n_1+\cdots +n_r)^{x+1}}\\
&=r \sum_{n_1,\ldots,n_r\ge 1}
\dfrac{1}{n_1\cdots n_{r-1}(n_1+\cdots +n_r)^{x+1}}.
\end{align*}
Similarly, we have
\begin{align*}
&M_r(x;(\underbrace{1,\ldots,1}_{r}),0)\\
&=r \sum_{n_1,\ldots,n_r\ge 1}
\dfrac{n_1+\cdots +n_{r-1}}
{n_1\cdots n_{r-1}(n_1+\cdots+n_{r-1})(n_1+\cdots +n_r)^{x+1}}\\
&=r(r-1) \sum_{n_1,\ldots,n_r\ge 1}
\dfrac{1}
{n_1\cdots n_{r-2}(n_1+\cdots+n_{r-1})(n_1+\cdots +n_r)^{x+1}}.
\end{align*}
Repeating this procedure, we obtain
\begin{align*}
&M_r(x;(\underbrace{1,\ldots,1}_{r}),0)\\
&=r!
\sum_{n_1,n_2,\ldots,n_r\ge 1}
\dfrac{1}{n_1(n_1+n_2) \cdots (n_1+\cdots +n_{r-1})
(n_1+\cdots +n_r)^{x+1}}.
\end{align*}
Finally, the right-hand side can be expressed as 
\[
r!
\sum_{1\le m_1<m_2<\cdots <m_r}
\dfrac{1}{m_1m_2\cdots m_{r-1}m_r^{x+1}}
=r! \zeta_{EZ,r}(\underbrace{1,\ldots,1}_{r-1},x+1).
\]
This confirms \eqref{eq:M-MZF}.
\end{proof}

Combining Proposition \ref{M-MZF} and Theorem \ref{asymptotic for M},
we immediately obtain \eqref{eq:asymptotic for MZF}.

\section{Complete asymptotic expansion via the refined method}
\label{sec:refinement}

Now we proceed to more elaborate studies on $I_r(x;\bm{\omega}, a)$. Our main aim in this section is to prove Theorem \ref{v9 first complete asymptotic expansion of integral}. For that purpose, we have to consider the following series. Let $(x)_r$ denote the Pochhammer symbol given by $(x)_r=x(x+1)\cdots (x+r-1)$
for $r\in\Zb_{\ge 1}$, and then define
\begin{equation}
	\label{def of S}
	S_r(x,\bm{\omega})
	=\sum_{k_1,\ldots,k_r\ge 1}
	\dfrac{(x)_{k_1+\cdots+k_r} \omega_1^{k_1}\cdots \omega_r^{k_r}}
	{k_1\cdots k_r \cdot k_1!\cdots k_r!}
\end{equation}
for $r\in\Zb_{\ge 1}$, $x>0$,
and $\bm{\omega}=(\omega_1,\ldots,\omega_r) \in \Cb^r$
with $|\bm{\omega}|<1$.
We first show

\begin{proposition}
The series $S_r(x,\bm{\omega})$ defined by \eqref{def of S} converges absolutely.
\end{proposition}

\begin{proof}
Using the multinomial theorem, we have
\begin{align*}
\sum_{k_1,\ldots,k_r\ge 1}
\dfrac{(x)_{k_1+\cdots+k_r} |\omega_1|^{k_1}\cdots |\omega_r|^{k_r}}
{k_1\cdots k_r \cdot k_1!\cdots k_r!}
&\le \sum_{m=r}^\infty
\dfrac{(x)_m}{m!}
\sum_{\substack{k_1,\ldots,k_r\ge 1\\ k_1+\cdots+k_r=m}}
\dfrac{m!}{k_1!\cdots k_r!} |\omega_1|^{k_1}\cdots |\omega_r|^{k_r}\\
&\le
\sum_{m=r}^\infty \dfrac{(x)_m |\bm{\omega}|^m}{m!}.
\end{align*}
Here we observe that
\begin{align}
\label{Pochhammer}
\dfrac{(x)_m}{m!}
&=\dfrac{x}{m} \left(1+\dfrac{x}{1}\right)
\left(1+\dfrac{x}{2}\right)\cdots \left(1+\dfrac{x}{m-1}\right)\\
&=\dfrac{x}{m} \exp\left(\sum_{h=1}^{m-1} \dfrac{x}{h}\right)
\prod_{h=1}^{m-1} \left(1+\dfrac{x}{h}\right) e^{-\frac{x}{h}}
\notag 
\end{align}
When $m\to\infty$, the product part is $O(1)$, and 
$$
\exp\left(\sum_{h=1}^{m-1} \dfrac{x}{h}\right)
=\left(\exp(\log m+O(1))\right)^x \ll m^x,
$$
hence
\begin{align}\label{Pochhammer-2}
\dfrac{(x)_m}{m!}\ll m^{x-1}
\end{align}
as $m\to\infty$.
Therefore the above positive series is
$\displaystyle \ll \sum_{m=r}^\infty m^{x-1}|\bm{\omega}|^m<\infty$ when $|\bm{\omega}|<1$.
\end{proof}

%
%
Now the integral $I_r(x;\bm{\omega},a)$ can be expressed
as a finite sum involving series of the type \eqref{def of S} as follows.
\begin{theorem}
	\label{expression by S}
	Assume that $\bm{\omega}$ is an $r$-tuple of positive real numbers
	with $|\bm{\omega}|<a$. Then
	\begin{equation}
		\label{eq:expression by S}
		I_r(x;\bm{\omega},a)
		=
		(-1)^r \dfrac{e^{-\gamma x}}{\Gamma(x)}
		\sum_{A,B,C}
		\biggl(\prod_{i\in A} \log \omega_i\biggr)
		\biggl(\dfrac{\partial}{\partial x}\biggr)^{|B|}
		\dfrac{\Gamma(x)}{a^x e^{-\gamma x}}
		S_{|C|}\biggl(x,-\dfrac{\bm{\omega}_C}a\biggr),
	\end{equation}
	where $A$, $B$, and $C$ are disjoint subsets of $R$ satisfying $A\sqcup B\sqcup C=R$
	and $S_0(x,\emptyset)=1$.
\end{theorem}

Note that
this Theorem \ref{expression by S} gives a refinement of Theorem \ref{asymptotic for I}.
In fact, since
\[
\biggl(\dfrac{\partial}{\partial x}\biggr)^{|B|}
\dfrac{\Gamma(x)}{a^x e^{-\gamma x}}
S_{|C|}\biggl(x,-\dfrac{\bm{\omega}_C}a\biggr)
=
\begin{cases}
	\dfrac{(-1)^{|B|}|B|!}{x^{|B|+1}}+O(1) & \mbox{if $C=\emptyset$},\\[8pt]
	O(1) & \mbox{if $C\neq\emptyset$},
\end{cases}
\]
we see that Theorem \ref{expression by S} implies
Theorem \ref{asymptotic for I} when $|\bm{\omega}|<a$.

\begin{proof}[Proof of Theorem \ref{expression by S}]
Applying \eqref{incomplete gamma} to the incomplete gamma functions
appearing on the right-hand side of \eqref{expression by incomplete Gamma},
we have
\begin{align*}
\prod_{i=1}^r \Gamma(0,\omega_iu)
&=(-1)^r \prod_{i=1}^r
\left(\log \omega_i+\gamma+\log u +
\sum_{k=1}^\infty \dfrac{(-\omega_i u)^k}{k\cdot k!}\right)\\
&=(-1)^r \sum_{A,B,C}
\biggl(\prod_{i\in A} \log \omega_i \biggr)
(\gamma+\log u)^{|B|}
\prod_{j\in C} \sum_{k_j=1}^\infty
\dfrac{(-\omega_ju)^{k_j}}{k_j\cdot k_j!}.
\end{align*}
Substituting this into \eqref{expression by incomplete Gamma},
we have
\begin{align*}
&I_r(x;\bm{\omega},a)\\
&=\dfrac{(-1)^r}{\Gamma(x)}
\int_0^\infty
\biggl[\sum_{A,B,C} \biggl(\prod_{i\in A} \log \omega_i \biggr)
(\gamma+\log u)^{|B|}
\biggl(
\prod_{j\in C} \sum_{k_j=1}^\infty
\dfrac{(-\omega_ju)^{k_j}}{k_j\cdot k_j!}
\biggr)\biggr]
e^{-au} u^{x-1}du.
\end{align*}
We can show that it is possible to interchange the order of integration and summation
here.    In fact, since
\[
\sum_{k_j=1}^\infty
\left|\dfrac{(-\omega_ju)^{k_j}}{k_j\cdot k_j!}\right|
\le \sum_{k_j=1}^\infty \dfrac{(\omega_ju)^{k_j}}{k_j!}
\le e^{\omega_j u},
\]
we can evaluate the integral
\begin{align*}
&\int_0^\infty
|\gamma+\log u|^{|B|}
\left(
\prod_{j\in C} \sum_{k_j=1}^\infty
\left|\dfrac{(-\omega_ju)^{k_j}}{k_j\cdot k_j!}\right|
\right)
e^{-au} u^{x-1}du\\
&\le
\int_0^\infty (\gamma+|\log u|)^{|B|} e^{-(a-|\bm{\omega}_C|)u} u^{x-1}du\\
&<\infty
\end{align*}
because  $a>|\bm{\omega}|\ge |\bm{\omega}_C|$.
This confirms the validity of interchanging integration and summation.

As a result, we encounter the integral
\[
\int_0^\infty (\gamma+\log u)^{|B|} e^{-au} u^{x+\sum_{j\in C} k_j-1} du,
\]
which is equal to
\begin{align*}
&e^{-\gamma x} \int_0^\infty
\left(\dfrac{\partial}{\partial x}\right)^{|B|}
e^{-au} e^{\gamma x} u^{x+\sum_{j\in C} k_j-1}du\\
&=e^{-\gamma x} 
\left(\dfrac{\partial}{\partial x}\right)^{|B|}
e^{\gamma x} \int_0^\infty e^{-au} u^{x+\sum_{j\in C} k_j-1}du\\
&=e^{-\gamma x} 
\left(\dfrac{\partial}{\partial x}\right)^{|B|}
e^{\gamma x}
\dfrac{\Gamma(x+\sum_{j\in C} k_j)}{a^{x+\sum_{j\in C} k_j}}.
\end{align*}
Hence we obtain
\begin{align*}
I_r(x;\bm{\omega},a)
&=(-1)^r \dfrac{e^{-\gamma x}}{\Gamma(x)}
\sum_{A,B,C} \left(\prod_{i\in A} \log \omega_i\right)\\
&\qquad \times
\sum_{k_{j_1},\ldots,k_{j_s}=1}^\infty
\left(\dfrac{\partial}{\partial x}\right)^{|B|}
\dfrac{\Gamma(x)}{a^x e^{-\gamma x}}
\dfrac{(x)_{k_{j_1}+\cdots +k_{j_s}} (-\omega_{j_1}/a)^{k_{j_1}} \cdots (-\omega_{j_s}/a)^{k_{j_s}}}
{k_{j_1}\cdots k_{j_s} \cdot k_{j_1}!\cdots k_{j_s}!},
\end{align*}
where $s=|C|$ and $C=\{j_1,\ldots,j_s\}$.
Therefore, to prove Theorem \ref{expression by S},
it is sufficient to demonstrate that differentiation and summation can be interchanged.
Moreover, applying the Leibniz rule,
the remaining task is to show that
the series
\begin{equation}
\label{eq:interchange}
\sum_{n= |B|+1}^\infty
\left(\dfrac{\partial}{\partial x}\right)^{b}(x)_{n}
\sum_{\substack{k_{j_1},\ldots,k_{j_s}\ge1\\ k_{j_1}+\cdots+k_{j_s}=n}}
\dfrac{(-\omega_{j_1}/a)^{k_{j_1}} \cdots (-\omega_{j_s}/a)^{k_{j_s}}}
{k_{j_1}\cdots k_{j_s} \cdot k_{j_1}!\cdots k_{j_s}!},
\end{equation}
where $b=1,2,\ldots,|B|$, converges and that
it is possible to interchange differentiation and summation for
\eqref{eq:interchange}.

Let $\alpha$ be any positive real number and suppose $0<x<\alpha$.
For $n\ge b+1$, we have
\begin{align*}
\left(\dfrac{\partial}{\partial x}\right)^b
(x)_{n}
&\le
n(n-1)\cdots (n-b+1) \cdot (x+b)_{n-b}\\
&\ll n! (n-b)^{\alpha+b-1},
\end{align*}
because the first inequality can be shown by induction on $b$, while the second inequality
follows from \eqref{Pochhammer-2}.
Therefore
\begin{align*}
&\sum_{n= |B|+1}^\infty
\biggl|
\left(\dfrac{\partial}{\partial x}\right)^{b}(x)_{n}
\sum_{\substack{k_{j_1},\ldots,k_{j_s}\ge1\\ k_{j_1}+\cdots+k_{j_s}=n}}
\dfrac{(-\omega_{j_1}/a)^{k_{j_1}} \cdots (-\omega_{j_s}/a)^{k_{j_s}}}
{k_{j_1}\cdots k_{j_s} \cdot k_{j_1}!\cdots k_{j_s}!}\biggr|\\
&\ll
\sum_{n=|B|+1}^\infty (n-b)^{\alpha+b-1}
\sum_{\substack{k_{j_1},\ldots,k_{j_s}\ge 1\\k_{j_1}+\cdots+k_{j_s}=n}}
\dfrac{n!}{k_{j_1}\cdots k_{j_s} \cdot k_{j_1}!\cdots k_{j_s}!}
\left(\dfrac{\omega_{j_1}}{a}\right)^{k_{j_1}}
\cdots
\left(\dfrac{\omega_{j_s}}{a}\right)^{k_{j_s}}\\
&\le
\sum_{n=|B|+1}^\infty (n-b)^{\alpha+b-1}
\left(\dfrac{|\bm{\omega}_C|}{a}\right)^n\\
&<\infty.
\end{align*}
This estimate verifies that \eqref{eq:interchange} converges,
and allows differentiation and summation to be performed in either order.
This completes the proof of Theorem \ref{expression by S}.
\end{proof}

To deduce the complete asymptotic expansion for $I_r(x;\bm{\omega},a)$
from Theorem \ref{expression by S},
we have to investigate the power series of $S_r(x,\bm{\omega})$ and its coefficients.

\begin{proposition}
\label{power series of S}
$S_r(x,\bm{\omega})$ can be expressed as a power series in $x$ of the form
\begin{equation}
\label{eq:power series for S}
S_r(x,\bm{\omega})=\sum_{l=1}^\infty T_{r,l}(\bm{\omega}) x^l
\end{equation}
for $r\in\Zb_{\ge 1}$, $x>0$, and $\bm{\omega}\in\Cb^r$ with $|\bm{\omega}|<1$,
where
\[
T_{r,l}(\bm{\omega})
=\sum_{m=l}^\infty
\begin{bmatrix}
m\\
l
\end{bmatrix}
\sum_{\substack{k_1,\ldots,k_r\ge 1 \\ k_1+\cdots+k_r=m}}
\dfrac{\omega_1^{k_1}\cdots \omega_r^{k_r}}{k_1\cdots k_r \cdot k_1!\cdots k_r!}
\]
and
$\displaystyle
\begin{bmatrix}
m\\
l
\end{bmatrix}$
is the Stirling number of the first kind.
\end{proposition}

\begin{proof}
Using the relation
\begin{equation}
\label{eq:Pochhammer-Stirling}
(x)_m=\sum_{l=1}^m
\begin{bmatrix}
m\\
l
\end{bmatrix}
x^l
\end{equation}
for $m\ge 1$,
we have
\begin{align*}
S_r(x,\bm{\omega})
&=\sum_{m=r}^\infty
\sum_{\substack{k_1,\ldots,k_r\ge 1\\ k_1+\cdots +k_r=m}}
\left(\sum_{l=1}^m
\begin{bmatrix}
m\\
l
\end{bmatrix}
x^l\right)
\dfrac{\omega_1^{k_1}\cdots \omega_r^{k_r}}
{k_1\cdots k_r\cdot k_1!\cdots k_r!}\\
&=
\sum_{l=1}^\infty x^l
\sum_{m=\max\{r,l\}}^\infty
\begin{bmatrix}
m\\
l
\end{bmatrix}
\sum_{\substack{k_1,\ldots,k_r\ge 1\\ k_1+\cdots+k_r=m}}
\dfrac{\omega_1^{k_1}\cdots \omega_r^{k_r}}
{k_1\cdots k_r\cdot k_1!\cdots k_r!}.
\end{align*}
If $m<r$, then the innermost sum is empty and can be regarded as $0$.
Hence we obtain \eqref{eq:power series for S}
\end{proof}

\begin{proposition}
\label{expression of T}
We have
$T_{1,l}(\omega)=\Li _{\underbrace{\mbox{\scriptsize $1,\ldots,1$}}_{l-1},2} (\omega)$
and 
\begin{equation}
\label{eq:T-MPL}
T_{r,l}(\bm{\omega})
=\int_0^{\omega_1}\cdots\int_0^{\omega_r}
\sum_{J\subset R} (-1)^{r-|J|} \Li_{\underbrace{\mbox{\scriptsize $1,\ldots,1$}}_{l}}\biggl(\sum_{j\in J} t_j\biggr)
\dfrac{dt_1\cdots dt_r}{t_1\cdots t_r}
\end{equation}
for $r\ge 1$, where $\Li_{k_1,k_2,\ldots,k_l}(z)$ is the one-variable multiple polylogarithm defined by \eqref{eq:def of one-variable MPL}.
\end{proposition}

Remark that the integrand of \eqref{eq:T-MPL} can be expressed
in terms of the classical logarithm by using the relation
\[
\Li_{\underbrace{\mbox{\scriptsize $1,\ldots,1$}}_{l}}(z)
=\dfrac{1}{l!} (-\log (1-z))^l
\]
(see \cite[Lemma 1(ii)]{AK99}).

\begin{proof}[Proof of Proposition \ref{expression of T}]
It follows from \eqref{eq:Pochhammer-Stirling} and the first line in \eqref{Pochhammer} that
\begin{equation}
\label{eq:Stirling-Harmonic}
\dfrac{1}{m!}
\begin{bmatrix}
m\\
 l
\end{bmatrix}
=\sum_{1\le m_1<\cdots<m_{l-1}<m} \dfrac{1}{m_1\cdots m_{l-1}m}.
\end{equation}
Hence we obtain
\[
T_{1,l}(\omega)
=\sum_{m=l}^\infty
\begin{bmatrix}
m\\
l
\end{bmatrix}
\dfrac{\omega^m}{m\cdot m!}
=\sum_{m=l}^\infty \sum_{1\le m_1<\cdots<m_{l-1}<m}
\dfrac{\omega^m}{m_1\cdots m_{l-1}m^2}
=\Li _{\underbrace{\mbox{\scriptsize $1,\ldots,1$}}_{l-1},2} (\omega).
\]
In general, using the multinomial theorem we have
\begin{align*}
T_{r,l}(\bm{\omega})
&=
\int_0^{\omega_1}\cdots \int_0^{\omega_r}
\sum_{m=l}^\infty
\begin{bmatrix}
m\\
l
\end{bmatrix}
\sum_{\substack{k_1,\ldots,k_r\ge1\\ k_1+\cdots +k_r=m}}
\dfrac{t_1^{k_1-1}\cdots t_r^{k_r-1}}{k_1!\cdots k_r!} dt_1\cdots dt_r\\
&=
\int_0^{\omega_1}\cdots \int_0^{\omega_r}
\sum_{m=l}^\infty
\dfrac{1}{m!}
\begin{bmatrix}
m\\
l
\end{bmatrix}
\sum_{J\subset R} (-1)^{r-|J|}\biggl(\sum_{j\in J} t_j\biggr)^m
\dfrac{dt_1\cdots dt_r}{t_1\cdots t_r}.
\end{align*}
Using \eqref{eq:Stirling-Harmonic} again,
we get \eqref{eq:T-MPL}.
\end{proof}

We are now ready to present the proof of Theorem \ref{v9 first complete asymptotic expansion of integral}, in which Theorem \ref{expression by S} and Proposition \ref{power series of S} play a central role.
\begin{proof}[Proof of Theorem \ref{v9 first complete asymptotic expansion of integral}]
Let $f(x)=\Gamma(x+1)/e^{-\gamma x}=x\Gamma(x)/e^{-\gamma x}$. Then from Theorem \ref{expression by S}, we have
		\begin{align*}
			a^x I_r(x;\bm{\omega},a)
			&=I_r(x;\bm{\omega}/a,1)\\
			&=(-1)^r \dfrac{e^{-\gamma x}}{\Gamma(x)}
			\sum_{A,B,C}
			\biggl(\prod_{i\in A} \log \dfrac{\omega_i}{a}\biggr)
			\biggl(\dfrac{\partial}{\partial x}\biggr)^{|B|}
			\dfrac{\Gamma(x)}{e^{-\gamma x}}
			S_{|C|}\biggl(x,-\dfrac{\bm{\omega}_C}a\biggr)\\
			&=(-1)^r \dfrac{x}{f(x)} \sum_{\substack{A,B,C \\ C=\emptyset}}
			\biggl(\prod_{i\in A} \log \dfrac{\omega_i}{a}\biggr)
			\biggl(\dfrac{d}{d x}\biggr)^{|B|}
			\dfrac{f(x)}{x}\\
			&\quad
			+(-1)^r \dfrac{x}{f(x)} \sum_{\substack{A,B,C \\ C\ne\emptyset}}
			\biggl(\prod_{i\in A} \log \dfrac{\omega_i}{a}\biggr)
			\biggl(\dfrac{\partial}{\partial x}\biggr)^{|B|}
			\dfrac{f(x)}{x} S_{|C|}\biggl(x,-\dfrac{\bm{\omega}_C}a\biggr)\\
			&=:S_1+S_2\,,
		\end{align*}
		where $S_1$ and $S_2$ denote the sums over $C=\emptyset$ and $C\ne\emptyset$, respectively.

		We first consider $S_1$. By using the result \eqref{v9 a(k,m) sequence},  we compute
		\begin{align*}
			\dfrac{x}{f(x)}
			\biggl(\dfrac{d}{d x}\biggr)^{|B|} \dfrac{f(x)}{x}
			&=
			\dfrac{x}{f(x)} \sum_{k=0}^{|B|} \binom{|B|}{k}
			\biggl(\dfrac{d}{d x}\biggr)^{k} f(x) \cdot
			\biggl(\dfrac{d}{dx}\biggr)^{|B|-k} \dfrac1x\\
			&=
			x \sum_{k=0}^{|B|} \binom{|B|}{k}
			\biggl(\sum_{l=0}^\infty a_{k,l}x^l\biggr)
			\cdot (-1)^{|B|-k} \dfrac{(|B|-k)!}{x^{|B|-k+1}}\\
			&=
			|B|! \sum_{k=0}^{|B|} \sum_{l=0}^\infty
			(-1)^{|B|-k}  \frac{a_{k,l}}{k!} x^{l-|B|+k}\\
			&=\sum_{n=-|B|}^\infty
			\biggl(|B|!
			\sum_{k=0}^{\min\{|B|,|B|+n\}}
			(-1)^{|B|-k} \dfrac{a_{k,|B|+n-k}}{k!}\biggr)
			x^n.
		\end{align*}
		Hence we get (with putting $|A|=j$ and $|B|=r-j$)
		\begin{align*}
			S_1&=(-1)^r \sum_{j=0}^r \Lambda_j(\bm{\omega}/a)
			\sum_{n=-r+j}^\infty
			\biggl(
			(r-j)! \sum_{k=0}^{\min\{r-j,r-j+n\}}
			(-1)^{r-j-k} \dfrac{a_{k,r-j+n-k}}{k!}
			\biggr)x^n\\
			&=(-1)^r \sum_{j=0}^r \Lambda_j(\bm{\omega}/a)
			\sum_{m=j}^\infty
			\biggl(
			(r-j)! \sum_{k=0}^{\min\{r-j,m-j\}}
			(-1)^{r-j-k} \dfrac{a_{k,m-j-k}}{k!}
			\biggr)x^{m-r}\\
			&=
			\sum_{m=0}^\infty \biggl(
			\sum_{j=0}^{\min\{r,m\}}
			(r-j)!
			\Lambda_j(\bm{\omega}/a)
			\sum_{k=0}^{\min\{r-j,m-j\}}
			(-1)^{j+k} \dfrac{a_{k,m-j-k}}{k!}
			\biggr)x^{m-r}.
		\end{align*}

		We next turn to $S_2$.
		We know from Proposition \ref{power series of S} that
		\[
		\dfrac{1}{x}S_{|C|}\biggl(x,-\dfrac{\bm{\omega}_C}{a}\biggr)
		=\sum_{l=1}^\infty T_{|C|,l}\biggl(-\dfrac{\bm{\omega}_C}{a}\biggr) x^{l-1}
		\]
		for  $C\ne\emptyset$, and then it implies
		\begin{align*}
			\biggl(\dfrac{\partial}{\partial x}\biggr)^{k}
			\dfrac{1}{x}S_{|C|}\biggl(x,-\dfrac{\bm{\omega}_C}{a}\biggr)
			&=\sum_{l=1}^\infty T_{|C|,l}\biggl(-\dfrac{\bm{\omega}_C}{a}\biggr)
			(l-1)(l-2)\cdots (l-k) x^{l-k-1}\\
			&=\sum_{n=0}^\infty T_{|C|,n+k+1}\biggl(-\dfrac{\bm{\omega}_C}{a}\biggr)
			(n+k)(n+k-1)\cdots (n+1) x^n\\
			&=\sum_{n=0}^\infty (n+1)_k T_{|C|,n+k+1}\biggl(-\dfrac{\bm{\omega}_C}{a}\biggr)
			x^n.
		\end{align*}
		Now again using the result \eqref{v9 a(k,m) sequence}, we calculate 
		\begin{align*}
			\lefteqn{\dfrac{x}{f(x)}
				\biggl(\dfrac{\partial}{\partial x}\biggr)^{|B|}
				\dfrac{f(x)}{x} S_{|C|}\biggl(x,-\dfrac{\bm{\omega}_C}a\biggr)}
			\quad & \\
			&= \dfrac{x}{f(x)} \sum_{k=0}^{|B|} \binom{|B|}{k}
			\biggl(\dfrac{d}{dx}\biggr)^{|B|-k} f(x) \cdot
			\biggl(\dfrac{d}{dx}\biggr)^{k}
			\dfrac{1}{x}S_{|C|}\biggl(x,-\dfrac{\bm{\omega}_C}{a}\biggr)\\
			&=x \sum_{k=0}^{|B|} \binom{|B|}{k}
			\biggl(\sum_{l=0}^\infty a_{|B|-k,l}x^l\biggr)
			\biggl(\sum_{n=0}^\infty (n+1)_{k} T_{|C|,n+k+1}\biggl(-\dfrac{\bm{\omega}_C}{a}\biggr)
			x^n \biggr)\\
			&=\sum_{m'=1}^\infty
			\biggl(\sum_{k=0}^{|B|} \binom{|B|}{k}
			\sum_{n=0}^{m'-1} a_{|B|-k,m'-n-1} (n+1)_{k} T_{|C|,n+k+1}\biggl(-\dfrac{\bm{\omega}_C}{a}\biggr)\biggr)
			x^{m'}\\
			&=\sum_{m=r+1}^\infty
			\biggl(\sum_{k=0}^{|B|} \binom{|B|}{k}
			\sum_{n=0}^{m-r-1} a_{|B|-k,m-r-n-1} (n+1)_{k} T_{|C|,n+k+1}\biggl(-\dfrac{\bm{\omega}_C}{a}\biggr)\biggr)
			x^{m-r}.
\end{align*}
Reindexing the innermost sum by setting $j=n+k+1$, we get
\begin{align*}
			\lefteqn{\dfrac{x}{f(x)}
				\biggl(\dfrac{\partial}{\partial x}\biggr)^{|B|}
				\dfrac{f(x)}{x} S_{|C|}\biggl(x,-\dfrac{\bm{\omega}_C}a\biggr)}
			& \\
			&=\sum_{m=r+1}^\infty
			\biggl(\sum_{k=0}^{|B|} \binom{|B|}{k}
			\sum_{j=k+1}^{m-r+k} a_{|B|-k,m-r-j+k} (j-k)_{k} T_{|C|,j}\biggl(-\dfrac{\bm{\omega}_C}{a}\biggr)
			\biggr)
			x^{m-r}\\
			&=\sum_{m=r+1}^\infty
			\biggl(\sum_{j=1}^{m-r+|B|}T_{|C|,j}\biggl(-\dfrac{\bm{\omega}_C}{a} \biggr)
			\sum_{k=\max\{0,j-m+r\}}^{\min\{|B|,j-1\}} \binom{|B|}{k}
			a_{|B|-k,m-r-j+k} (j-k)_{k}
			\biggr)
			x^{m-r}.
		\end{align*}
		Thus we obtain
		\begin{align*}
			S_2&=
			\sum_{m=r+1}^\infty \biggl(
			(-1)^r \sum_{\substack{A,B,C \\ C\ne\emptyset}}
			\biggl(\prod_{i\in A} \log \dfrac{\omega_i}{a}\biggr)
			\sum_{j=1}^{m-r+|B|}T_{|C|,j}\biggl(-\dfrac{\bm{\omega}_C}{a} \biggr)\\
			&\quad \times
			\sum_{k=\max\{0,j-m+r\}}^{\min\{|B|,j-1\}} \binom{|B|}{k}
			a_{|B|-k,m-r-j+k} (j-k)_{k}
			\biggr)
			x^{m-r}.
		\end{align*}
		Therefore the desired expressions for $d_{r,m}(\bm{\omega},a)$ follow from those for $S_1$ and $S_2$.
	\end{proof}
	Next, we analyze the numbers  $ \left(a_{k,m}\right)_{k,m\geq 0} $ in detail using the Fa\`a di Bruno formula.
We first define the $k$-th complete exponential Bell polynomial
$B_k(x_1,\ldots,x_k)$ by the recurrent relation
\begin{align}
	\label{recurrence relation for Bell polynomials}
	B_{k+1}(x_1,\ldots,x_{k+1})
=\sum_{s=0}^k \binom{k}{s} x_{s+1}B_{k-s}(x_1,\ldots,x_{k-s})\,\,\,\,(k=0,1,2,\ldots)
\end{align}
with $ B_0=1 $. For instance,
\[
B_1(x_1)=x_1,
\quad
B_2(x_1,x_2)=x_1^2+x_2,
\quad
B_3(x_1,x_2,x_3)=x_1^3+3x_1x_2+x_3.
\]
In \cite[\S4]{Bell34}, Bell gave its generating function
\begin{equation}
\label{eq:generating function}
\exp\biggl(\sum_{m=1}^\infty x_m \dfrac{t^m}{m!}\biggr)
=\sum_{k=0}^\infty B_k(x_1,\ldots,x_k)\dfrac{t^k}{k!}
\end{equation}
and its explicit expression
\[
B_k(x_1,\ldots,x_n)
=k! \sum_{\substack{s_1,s_2,\ldots,s_k\ge 0 \\ s_1+2s_2+\cdots +ks_k=k}}
\dfrac{x_1^{s_1} x_2^{s_2}\cdots x_k^{s_k}}{(1!)^{s_1}(2!)^{s_2}\cdots (k!)^{s_k} \cdot s_1! s_2! \cdots s_k!}
\]
for $k\ge 1$.

One special case of Fa\`a di Bruno formula for higher derivatives is the following:
	\begin{proposition}
		\label{v9 Faa di Bruno formula}
		For any function $h$ belonging to $C^{\infty}$-class,  
		\begin{align*}
			\left(\dfrac{d}{d x}\right)^{k}e^{h(x)}=e^{h(x)}B_k\left(h^{\prime}(x),h^{\prime \prime}(x), \ldots,h^{(k)}(x) \right).
		\end{align*}
	\end{proposition}
	It is known that, if $0<x<1$, then
	\begin{equation}
\label{eq:power series of log gamma}
		\dfrac{e^{-\gamma x}}{\Gamma(x+1)}=\exp\left(\sum_{m=1}^\infty\frac{(-1)^{m-1}\zeta(m)}{m}x^m\right).
	\end{equation}
	Note that throughout we interpret $ \zeta(1) $  as $ 0 $.
	Putting
	\begin{align*}
			h(x)=\sum_{m\geq 1}\frac{(-1)^{m}\zeta(m)}{m}x^m
		\end{align*}
    we have
    \begin{align*}
    \frac{\Gamma(x+1)}{e^{-\gamma x}}=e^{h(x)}.
    \end{align*}
	Hence from \eqref{v9 a(k,m) sequence} and Proposition \ref{v9 Faa di Bruno formula}, we have the following result:
	\begin{corollary}
		\label{v9 Bell polynomial and Coefficients}
		\begin{align*}
			\dfrac{e^{-\gamma x}}{\Gamma(x+1)} \biggl(\dfrac{d}{d x} \biggr)^{k}
			\dfrac{\Gamma(x+1)}{e^{-\gamma x}}=B_k\left(h^{\prime}(x),h^{\prime \prime}(x), \ldots,h^{(k)}(x) \right)
			=\sum_{m=0}^\infty a_{k,m} x^m.
		\end{align*}
	\end{corollary}
	\begin{example}
		\label{v9 example for a(k,m)}
	$ (1) $ 	For $ k=0 $, $ \displaystyle B_0=1=\sum_{m=0}^\infty a_{0,m} x^m $.
	It is clear that $a_{0,0}=1$ and $a_{0,m}=0\,\,(m\ge 1)$.\\[3pt]
	$ (2) $ For $ k=1 $,
	\begin{align*}
		B_1(h^\prime(x))=h^\prime(x)=\sum_{m=0}^\infty a_{1,m} x^m,\,\,\text{where}\,\,	h^\prime(x)=\sum_{m\geq 1}(-1)^{m+1}\zeta(m+1)x^m.
	\end{align*}
	Thus we obtain $ a_{1,0}=0\,,\,a_{1,m}=(-1)^{m+1} \zeta(m+1)\,\,(m=1,2,\ldots).$\\
	$ (3) $ For $ k=2$, $ B_2(h^{\prime}(x),h^{\prime \prime}(x) )=\left(h^\prime(x)\right)^2+h^{\prime \prime}(x)=\displaystyle \sum_{m=0}^\infty a_{2,m} x^m,$
	and 
	\begin{align*}
		&\left(h^\prime(x)\right)^2+h^{\prime \prime}(x)\\&\hspace{4mm}=\left(\sum_{m\geq 1}(-1)^{m+1}\zeta(m+1)x^m\right)^2+\sum_{m\geq 0}(-1)^{m+2}\zeta(m+2)(m+1)x^m\\
		&\hspace{4mm}=\sum_{m\geq 2}(-1)^m\left(\sum_{n_1+n_2=m}\zeta(n_1+1)\zeta(n_2+1)\right)x^m+\sum_{m\geq 0}(-1)^{m}\zeta(m+2)(m+1)x^m\\
		&\hspace{4mm}=\zeta(2)+\sum_{m\geq 1}(-1)^m\left((m+1)\zeta(m+2)+\sum_{n_1+n_2=m}\zeta(n_1+1)\zeta(n_2+1)\right)x^m.
	\end{align*}
	Hence we have
	\begin{align*}
		a_{2,0}=\zeta(2),\,\,a_{2,m}=(-1)^m\left((m+1)\zeta(m+2)+\sum_{n_1+n_2=m}\zeta(n_1+1)\zeta(n_2+1)\right)
	\end{align*}
for $m\ge 1$.
\end{example}
	\par
	These numbers $ \left(a_{k,m}\right)_{k,m\geq 0} $ adhere to an elegant recurrence relation, which is the following.
	\begin{theorem} 
		\label{v9 recrrence relation for a(k,m)}
		For any non-negative integers $ k,m $, we have
		\begin{align*}
			a_{k+1,m}=\sum_{r=0}^{k}{k \choose r}(-1)^{r+1}\sum_{\substack{n_1+n_2=m\\ n_1,n_2\geq 0}}(-1)^{n_1}(n_1+1)_r \zeta(n_1+r+1)a_{k-r,n_2},
		\end{align*}
		where $ (n_1+1)_r $ is the Pochhammer symbol, and the numbers $ \left(a_{0,m}\right)_{m\geq 0} $ are given in Example \ref{v9 example for a(k,m)} $ (1) $.
	\end{theorem}
	\begin{proof}
		From the  recurrence relation of Bell polynomial (see Formula \eqref{recurrence relation for Bell polynomials}), we have
		\begin{align*}
		B_{k+1}\left(h^{\prime}(x), \ldots,h^{(k+1)}(x) \right)=\sum_{r=0}^{k}{k \choose r}h^{(r+1)}(x)B_{k-r}\left(h^{\prime}(x), \ldots,h^{(k-r)}(x) \right).
		\end{align*} 
		Then from Corollary \ref{v9 Bell polynomial and Coefficients}, it follows that
		\begin{align*}
			\hspace{2mm}\sum_{m=0}^\infty & a_{k+1,m} x^m\\
			&=\sum_{r=0}^{k}{k \choose r}\left(\frac{d}{dx}\right)^{r+1}\left(\sum_{m\geq 1}\frac{(-1)^{m}\zeta(m)}{m}x^m\right)\left(\sum_{m=0}^\infty a_{k-r,m} x^m\right)\\
			&=\sum_{r=0}^{k}{k \choose r}\left(\sum_{m\geq r+1}\frac{(-1)^{m}\zeta(m)}{m}m(m-1)\cdots (m-r)x^{m-r-1}\right)\left(\sum_{m=0}^\infty a_{k-r,m} x^m\right)\\
			&=\sum_{r=0}^{k}{k \choose r}\left(\sum_{m\geq 0}\frac{(-1)^{m+r+1}\zeta(m+r+1)}{m+r+1}(m+1)_{r+1}x^m\right)\left(\sum_{m=0}^\infty a_{k-r,m} x^m\right)\\
			&=\sum_{r=0}^{k}{k \choose r}(-1)^{r+1}\left(\sum_{m\geq 0}(-1)^{m}(m+1)_{r}\zeta(m+r+1)x^m\right)\left(\sum_{m=0}^\infty a_{k-r,m} x^m\right).
		\end{align*}
		Now applying the rule for the product of two power series, and then interchanging the summation, we have
		\begin{align*}
			\sum_{m=0}^\infty & a_{k+1,m} x^m\\
			&=\sum_{m=0}^{\infty}\left(\sum_{r=0}^{k}{k \choose r}(-1)^{r+1}\sum_{\substack{n_1+n_2=m\\ n_1,n_2\geq 0}}(-1)^{n_1}(n_1+1)_r \zeta(n_1+r+1)a_{k-r,n_2}\right)x^m.
		\end{align*}
		Then the theorem is followed by comparing the coefficients of the power series.
	\end{proof}
\par
Now using the values of $ a_{k,m} $ (from Example \ref{v9 example for a(k,m)})
and  $ T_{r,l}(\bm{\omega}) $ (from Proposition \ref{expression of T}) in Theorem \ref{v9 first complete asymptotic expansion of integral}, we compute the coefficients $ d_{r,m}(\bm{\omega},a) $ for specific cases.
\begin{example}
	\label{v9 example for coefficients d}
	$(1)$ Let $d_{1,m}:=d_{1,m}(\omega,a)$. Then
	\[
	d_{1,0}=1,
	\quad
	d_{1,1}=-\log \dfrac{\omega}{a},
	\quad
	d_{1,m}=(-1)^{m+1}\zeta(m)
	-\Li _{\underbrace{\mbox{\scriptsize $1,\ldots,1$}}_{m-2},2} \Bigl(-\dfrac{\omega}{a}\Bigr)
	\quad (m\ge 2).
	\]
	As a result, we obtain \eqref{eq:I1-2}.\\
	$(2)$ Let $d_{2,m}:=d_{2,m}((\omega_1,\omega_2),a)$.
	We have
	\begin{align*}
		d_{2,0}&=2,
		\quad
		d_{2,1}=-\log \dfrac{\omega_1\omega_2}{a^2},
		\quad
		d_{2,2}=-\dfrac{\pi^2}6+\log \dfrac{\omega_1}{a} \log \dfrac{\omega_2}{a} 
	\end{align*}	
		because $a_{1,1}= a_{2,0}=\dfrac{\pi^2}{6}$, and
	\begin{align*}
		d_{2,m}&=(-1)^{m+1}2\zeta(m)+a_{2,m-2}
		+(-1)^{m-1}\zeta(m-1) \log \dfrac{\omega_1\omega_2}{a^2}\\
		&\quad 
		+T_{2,m-2}\biggl(\!\!\Bigl(-\dfrac{\omega_1}{a},-\dfrac{\omega_2}{a}\Bigr)\!\!\biggr)
		+\log \dfrac{\omega_1}{a}
		\Li _{\underbrace{\mbox{\scriptsize $1,\ldots,1$}}_{m-3},2} \Bigl(-\dfrac{\omega_2}{a}\Bigr)
		+\log \dfrac{\omega_2}{a}
		\Li _{\underbrace{\mbox{\scriptsize $1,\ldots,1$}}_{m-3},2} \Bigl(-\dfrac{\omega_1}{a}\Bigr)\\
		&\quad
		+(-1)^{m-2}\zeta(m-2)
		\biggl(
		\Li_2\Bigl(-\dfrac{\omega_1}{a}\Bigr)+\Li_2\Bigl(-\dfrac{\omega_2}{a}\Bigr)
		\biggr)\\
		&\quad
		+(m-2)
		\biggl(
		\Li _{\underbrace{\mbox{\scriptsize $1,\ldots,1$}}_{m-2},2} \Bigl(-\dfrac{\omega_1}{a}\Bigr)
		+
		\Li _{\underbrace{\mbox{\scriptsize $1,\ldots,1$}}_{m-2},2} \Bigl(-\dfrac{\omega_2}{a}\Bigr)
		\biggr)\\
		&\quad
		+\sum_{j=2}^{m-3}
		(-1)^{m-j-1}\zeta(m-j-1)
		\biggl(
		\Li _{\underbrace{\mbox{\scriptsize $1,\ldots,1$}}_{j-1},2} \Bigl(-\dfrac{\omega_1}{a}\Bigr)
		+
		\Li _{\underbrace{\mbox{\scriptsize $1,\ldots,1$}}_{j-1},2} \Bigl(-\dfrac{\omega_1}{a}\Bigr)
		\biggr)
	\end{align*}
	for $m\ge 3$, and where $ a_{2,m-2} $ is calculated in Example \ref{v9 example for a(k,m)} $ (3) $.
\end{example}

\section{Complete asymptotic expansion via an alternative method}
\label{sec:power series for I}
In this section we describe our second method, which leads to Theorem 
\ref{power series for I}.

In the proof of Theorem \ref{power series for I},
it is crucial to express the integral $I_r(x;\bm{\omega},a)$ as a finite sum of multiple polylogarithms of Hurwitz type,
as detailed in Propositions \ref{relation between I and Li0} and \ref{I-Li1} below.
To establish these results,
we first show the following related lemmas.

\begin{lemma}
For $r\in\Zb_{\ge 1}$, we have
\begin{equation}
\label{eq:relation 1}
I_r(x;\bm{\omega},a)
=\dfrac1x \sum_{i=1}^r I_{r-1}(x;(\omega_1,\ldots,\hat{\omega_i},\ldots,\omega_r),a+\omega_i)
+aI_r(x+1;\bm{\omega},a),
\end{equation}
where the hat notation means that the element with hat is omitted
and $I_0(x;\emptyset,a)=a^{-x}$.
\end{lemma}

\begin{proof}
From the definition of $I_r(x;\bm{\omega},a)$, we have
\begin{align*}
I_r(x;\bm{\omega},a)
&=\int_1^\infty \cdots \int_1^\infty
\dfrac{\omega_1t_1+\cdots +\omega_rt_r+a}{t_1\cdots t_r(\omega_1t_1+\cdots +\omega_rt_r+a)^{x+1}} \,
dt_1\cdots dt_r\\
&=
\sum_{i=1}^r
\int_1^\infty \cdots \int_1^\infty
\dfrac{\omega_i}{t_1\cdots\hat{t_i}\cdots t_r(\omega_1t_1+\cdots +\omega_rt_r+a)^{x+1}} \, 
dt_1 \cdots dt_r\\
&\quad
+a I_r(x+1;\bm{\omega},a).
\end{align*}
Since
\[
\int_1^\infty \dfrac{\omega_i dt_i}{(\omega_1t_1+\cdots+\omega_rt_r+a)^{x+1}}
=\dfrac{1}{x} \cdot \dfrac{1}{(\omega_1t_1+\cdots +\widehat{\omega_it_i}+\cdots+\omega_rt_r+a+\omega_i)^x},
\]
we obtain \eqref{eq:relation 1}.
\end{proof}

\begin{lemma}
For $r\in\Zb_{\ge 0}$, we have
\begin{equation}
\label{eq:relation 2}
I_r(x;\bm{\omega},a)
=\sum_{j=0}^r \dfrac{j!}{x^j}
\sum_{\substack{J\subset R \\ |J|=j}}
(a+|\bm{\omega}_J|) I_{r-j} (x+1;\bm{\omega}_{R\setminus J},a+|\bm{\omega}_J|).
\end{equation}
\end{lemma}

\begin{proof}
Since \eqref{eq:relation 2} obviously holds for $r=0$,
we assume $r\ge 1$ below.
We will prove, by induction on $k$, that
the relation
\begin{align}
\label{eq:ind hyp 1}
I_r(x;\bm{\omega},a)
&=\dfrac{k!}{x^k}\sum_{\substack{J\subset R\\ |J|=k}}
I_{r-k}(x;\bm{\omega}_{R\setminus J},a+|\bm{\omega}_J|)\\
&\quad
+\sum_{j=0}^{k-1} \dfrac{j!}{x^j}
\sum_{\substack{J\subset R\\ |J|=j}} (a+|\bm{\omega}_J|)
I_{r-j}(x+1;\bm{\omega}_{R\setminus J}, a+|\bm{\omega}_J|) \notag
\end{align}
holds for $k=1,2,\ldots,r$.
First, (\ref{eq:ind hyp 1}) holds for $k=1$, because it coincides with (\ref{eq:relation 1}).
Next we assume that (\ref{eq:ind hyp 1}) is true for some fixed $k\, \le r-1$.
By applying (\ref{eq:relation 1}) to (\ref{eq:ind hyp 1}),
we have
\begin{align*}
I_r(x;\bm{\omega},a)
&=
\dfrac{k!}{x^k}
\sum_{\substack{J\subset R\\ |J|=k}}
\biggl\{
\dfrac{1}{x}\sum_{i\in R\setminus J}
I_{r-k-1}(x;\bm{\omega}_{R\setminus (J\sqcup \{i\})},a+|\bm{\omega}_{J\sqcup\{i\}}|)\\
&\hskip20mm
+(a+|\bm{\omega}_J|)I_{r-k}(x+1;\bm{\omega}_{R\setminus J},a+|\bm{\omega}_J|)
\biggr\}\\
&\quad
+\sum_{j=0}^{k-1} \dfrac{j!}{x^j}
\sum_{\substack{J\subset R\\ |J|=j}} (a+|\bm{\omega}_J|)
I_{r-j}(x+1;\bm{\omega}_{R\setminus J}, a+|\bm{\omega}_J|)\\
&=
\dfrac{k!}{x^{k+1}}
\sum_{\substack{J\subset R\\ |J|=k}}
\sum_{i\in R\setminus J}
I_{r-k-1}(x;\bm{\omega}_{R\setminus (J\sqcup \{i\})},a+|\bm{\omega}_{J\sqcup\{i\}}|)\\
&\quad
+\sum_{j=0}^{k} \dfrac{j!}{x^j}
\sum_{\substack{J\subset R\\ |J|=j}} (a+|\bm{\omega}_J|)
I_{r-j}(x+1;\bm{\omega}_{R\setminus J}, a+|\bm{\omega}_J|).
\end{align*}
The first term on the right-hand side is
\begin{align*}
&\dfrac{k!}{x^{k+1}}
\sum_{\substack{J'\subset R\\ |J'|=k+1}}
(k+1)
I_{r-k-1}(x;\bm{\omega}_{R\setminus J'},a+|\bm{\omega}_{J'}|)\\
&=\dfrac{(k+1)!}{x^{k+1}}
\sum_{\substack{J'\subset R\\ |J'|=k+1}}
I_{r-k-1}(x;\bm{\omega}_{R\setminus J'},a+|\bm{\omega}_{J'}|).
\end{align*}
Hence \eqref{eq:ind hyp 1} holds for $k+1$,
and consequently \eqref{eq:ind hyp 1} holds for $k=1,2,\ldots,r$.

By substituting $k=r$ into (\ref{eq:ind hyp 1}),
we obtain
\begin{align*}
I_r(x;\bm{\omega},a)
&=\dfrac{r!}{x^r} I_0(x;\omega_\emptyset, a+|\bm{\omega}|)
+\sum_{j=0}^{r-1} \dfrac{j!}{x^j}
\sum_{\substack{J\subset R\\ |J|=j}}
(a+|\bm{\omega}_J|) I_{r-j}(x+1;\bm{\omega}_{R\setminus J},a+|\bm{\omega}_J|).
\end{align*}
The first term on the right-hand side can be rewritten as
\[
\dfrac{r!}{x^r}I_0(x;\emptyset,a+|\bm{\omega}|)
=\dfrac{r!}{x^r}(a+|\bm{\omega}|)^{-x}
=\dfrac{r!}{x^r}(a+|\bm{\omega}|)I_0(x+1;\emptyset,a+|\bm{\omega}|).
\]
Thus \eqref{eq:relation 2} follows.
\end{proof}

\begin{lemma}
\label{I-N}
For positive integers $r$ and $N$, we have
\begin{align}
\label{eq:relation 3}
I_r(x;\bm{\omega},a)
&=\sum_{\substack{j_1,\ldots,j_N\ge 0 \\ j_1+\cdots+j_N\le r}}
\dfrac{j_1!\cdots j_N!}{x^{j_1}(x+1)^{j_2}\cdots (x+N-1)^{j_N}}\\
&\quad
\times \sum_{J_1,\ldots,J_N} (a+|\bm{\omega}_{J_1}|) (a+|\bm{\omega}_{J_1\sqcup J_2}|)
\cdots (a+|\bm{\omega}_{J_1\sqcup \cdots \sqcup J_N}|) \notag\\
&\quad 
\times
I_{r-j_1-\cdots -j_N}(x+N; \bm{\omega}_{R\setminus (J_1\sqcup \cdots \sqcup J_N)},
a+|\bm{\omega}_{J_1\sqcup \cdots \sqcup J_N}|), \notag
\end{align}
where $J_1,\ldots,J_N$ are disjoint subsets of $R$ satisfying $|J_i|=j_i$
for $i=1,2,\ldots,N$.
\end{lemma}

\begin{proof}
We use induction on $N$.
For $N=1$, \eqref{eq:relation 3} is true, because it equals (\ref{eq:relation 2}).
Assume that \eqref{eq:relation 3} holds for some fixed $N$.
Then, using (\ref{eq:relation 2}), we have
\begin{align*}
&I_{r-j_1-\cdots -j_N}(x+N;\bm{\omega}_{R\setminus(J_1\sqcup\cdots\sqcup J_N)},a+|\bm{\omega}_{J_1\sqcup\cdots\sqcup J_N}|)\\
&=\sum_{j_{N+1}=0}^{r-j_1-\cdots-j_N}
\dfrac{j_{N+1}!}{(x+N)^{j_{N+1}}}
\sum_{\substack{J_{N+1}\subset R\setminus(J_1\sqcup\cdots \sqcup J_N)\\|J_{N+1}|=j_{N+1}}}
(a+|\bm{\omega}_{J_1\sqcup\cdots\sqcup J_{N+1}}|)\\
&\quad \times
I_{r-j_1-\cdots-j_{N+1}}(x+N+1;\bm{\omega}_{R\setminus(J_1\sqcup\cdots\sqcup J_{N+1})},a+|\bm{\omega}_{J_1\sqcup\cdots\sqcup J_{N+1}}|).
\end{align*}
Substituting this into \eqref{eq:relation 3} for $N$,
we obtain \eqref{eq:relation 3} for $N+1$.
Hence, Lemma \ref{I-N} is true.
\end{proof}

For $s=0,1,\ldots,r$,
let $T_N(s)$ denote the sum of the terms where $j_1+\cdots +j_N=s$ in \eqref{eq:relation 3}.
Namely,
\begin{align}
\label{def of T}
T_N(s)
&=\sum_{\substack{j_1,\ldots,j_N\ge 0 \\ j_1+\cdots+j_N=s}}
\dfrac{j_1!\cdots j_N!}{x^{j_1}(x+1)^{j_2}\cdots (x+N-1)^{j_N}}\\
&\quad \times
\sum_{J_1,\ldots,J_N} (a+|\bm{\omega}_{J_1}|) (a+|\bm{\omega}_{J_1\sqcup J_2}|)
\cdots (a+|\bm{\omega}_{J_1\sqcup \cdots \sqcup J_N}|) \notag\\
&\quad \times
I_{r-s}(x+N; \bm{\omega}_{R\setminus (J_1\sqcup \cdots \sqcup J_N)},
a+|\bm{\omega}_{J_1\sqcup \cdots \sqcup J_N}|), \notag
\end{align}
where $J_1,\ldots,J_N$ are disjoint subsets of $R$ satisfying $|J_i|=j_i$
for $i=1,2,\ldots,N$.

\begin{lemma}
\label{lim_N}
If $0\le s\le r-1$, then $\displaystyle \lim_{N\rightarrow \infty} T_N(s)=0$.
\end{lemma}

\begin{proof}
If $x>r\ge 1$, then
\begin{align*}
I_r(x;\bm{\omega},a)
&\le 
\int_1^\infty \cdots \int_1^\infty 
\dfrac{1}{(\omega_1t_1+\cdots +\omega_rt_r+a)^x} \, dt_1\cdots dt_r\\
&=
\dfrac{(a+|\bm{\omega}|)^{-x+r}}{\omega_1\cdots \omega_r (x-1)(x-2)\cdots (x-r)}.
\end{align*}
Hence  we obtain
\[
I_{r-s}(x+N; \bm{\omega}_{R\setminus (J_1\sqcup \cdots \sqcup J_N)},
a+|\bm{\omega}_{J_1\sqcup \cdots \sqcup J_N}|)
\ll
\dfrac{(a+|\bm{\omega}|)^{-N}}{N^{r-s}}
\]
as $N\to \infty$.
Moreover, we see that 
$x^{j_1}(x+1)^{j_2}\cdots (x+N-1)^{j_N}\ge x^s\gg 1$ and
\begin{align}\label{koremotuika}
(a+|\bm{\omega}_{J_1}|) (a+|\bm{\omega}_{J_1\sqcup J_2}|)
\cdots (a+|\bm{\omega}_{J_1\sqcup \cdots \sqcup J_N}|)
\le (a+|\bm{\omega}'|)^N,
\end{align}
where $\bm{\omega}'=(\omega_1,\ldots,\hat{\omega}_{i_0},\ldots,\omega_r)$ with
$\omega_{i_0}=\min_i \omega_i$, hence
$|\bm{\omega}'|=|\bm{\omega}|-\omega_{i_0}$.
The inequality \eqref{koremotuika} holds because $|J_1\sqcup\cdots \sqcup J_N|=s<r=|R|$.
Combining the above estimates, we obtain
\begin{equation}
T_N(s)
\ll \dfrac{1}{N^{r-s}}
\left(\dfrac{a+|\bm{\omega}'|}{a+|\bm{\omega}|}\right)^N
\sum_{\substack{j_1,\ldots,j_N\ge 0 \\ j_1+\cdots+j_N=s}}
j_1!\cdots j_N!
\sum_{J_1,\ldots,J_N}1.
\end{equation}
Also, since
\[
\sum_{\substack{j_1,\ldots,j_N\ge 0 \\ j_1+\cdots+j_N=s}}
j_1!\cdots j_N!
\sum_{J_1,\ldots,J_N}1
=\frac{r!}{(r-s)!} \binom{s+N-1}{s}
\ll N^s,
\]
we have
\[
T_N(s)\ll N^{2s-r} \left(\dfrac{a+|\bm{\omega}'|}{a+|\bm{\omega}|}\right)^N
\to 0
\]
as $N\to \infty$, because $|\bm{\omega}'|<|\bm{\omega}|$.
This completes the proof of Lemma \ref{lim_N}.
\end{proof}

We introduce two types of multiple polylogarithms of Hurwitz type as follows:
\begin{align}
\Li_{\bm{k}}^0(x;\bm{z})
&=\sum_{0\le n_1<n_2<\cdots <n_r}
\dfrac{z_1^{n_1}z_2^{n_2}\cdots z_r^{n_r}}
{(n_1+x)^{k_1}(n_2+x)^{k_2}\cdots (n_r+x)^{k_r}},
\label{eq:Li0}\\
\Li_{\bm{k}}^1(x;\bm{z})
&=\sum_{1\le n_1<n_2<\cdots <n_r}
\dfrac{z_1^{n_1}z_2^{n_2}\cdots z_r^{n_r}}
{(n_1+x)^{k_1}(n_2+x)^{k_2}\cdots (n_r+x)^{k_r}},
\label{eq:Li1}
\end{align}
where $r\in\Zb_{\ge 1}$, $x>0$, $\bm{k}=(k_1,\ldots,k_r)\in(\Zb_{\ge 1})^r$,
and $\bm{z}=(z_1,\ldots,z_r)\in\Cb^r$ with $|z_i|<1$.
Here, $z_1^{n_1}$ is interpreted as $1$ when $z_1=0$ and $n_1=0$.
We now express $(a+|\bm{\omega}|)^x I_r(x;\bm{\omega},a)$
as a finite sum of functions of the form \eqref{eq:Li0}.

\begin{proposition}
\label{relation between I and Li0}
For $r\in\Zb_{\ge 1}$, we have
\begin{align}
\label{eq:relation between I and Li0}
&(a+|\bm{\omega}|)^x I_r(x;\bm{\omega},a)\\
&=
\sum_{s=1}^r \sum_{\substack{\bm{k}\in (\Zb_{\ge 1})^s \\ k_1+\cdots +k_s=r}}
k_1!\cdots k_s! \notag\\
&\quad \times
\sum_{K_1,\ldots,K_s}
\Li_{\bm{k}}^0
\biggl(x;
\dfrac{a}{a+|\bm{\omega}_{K_1}|},
\dfrac{a+|\bm{\omega}_{K_1}|}{a+|\bm{\omega}_{K_1\sqcup K_2}|},
\ldots,
\dfrac{a+|\bm{\omega}_{K_1\sqcup \cdots \sqcup K_{s-1}}|}
{a+|\bm{\omega}|}
\biggr), \notag
\end{align}
where $K_1,\ldots,K_s$ are disjoint subsets of $R$ satisfying $|K_i|=k_i$
for $i=1,2,\ldots,s$.
\end{proposition}

\begin{proof}
It is evident from Lemmas \ref{I-N} and \ref{lim_N} that
$\displaystyle I_r(x;\bm{\omega},a)=\lim_{N\rightarrow \infty} T_N(r)$. Thus, in what follows, we evaluate $T_N(r)$.
Note that
\begin{align*}
T_N(r)
&=\sum_{\substack{j_1,j_2,\ldots,j_N\ge 0 \\ j_1+j_2+\cdots+j_N=r}}
\dfrac{j_1!j_2!\cdots j_N!}
{x^{j_1}(x+1)^{j_2}\cdots (x+N-1)^{j_N}}\\
&\quad
\times \sum_{J_1,\ldots,J_N}
(a+|\bm{\omega}_{J_1}|)\cdots (a+|\bm{\omega}_{J_1\sqcup\cdots\sqcup J_N}|)
\cdot (a+|\bm{\omega}|)^{-x-N},
\end{align*}
where $J_1,\ldots,J_N$ are disjoint subsets of $R$ satisfying $|J_i|=j_i$ for $i=1,2,\ldots,N$.
For each $(j_1,j_2,\ldots,j_N)$,
we set 
$\{n_1,n_2,\ldots,n_s\}=\{n\in \{0,1,\ldots,N-1\} \mid j_{n+1}\ne 0\}$
where $0\le n_1<n_2<\cdots <n_s<N$.
It is clear that $1\le s\le r$.
We put $k_i=j_{n_i+1}$ and $K_i=J_{n_i+1}$
for $i=1,2,\ldots,s$.
Then we see that $k_1,\ldots,k_s\in \Zb_{\ge 1}$ with $k_1+\cdots+k_s=r$
and that $K_1,\ldots,K_s$ are disjoint subsets of $R$ satisfying $|K_i|=k_i$
for $i=1,2,\ldots,s$.
Moreover, we have
\begin{align*}
&\dfrac{j_1!j_2!\cdots j_N!}
{x^{j_1}(x+1)^{j_2}\cdots (x+N-1)^{j_N}}\\
&\quad \times \sum_{J_1,\ldots,j_N}
(a+|\bm{\omega}_{J_1}|)\cdots (a+|\bm{\omega}_{J_1\sqcup\cdots\sqcup J_N}|)
\cdot (a+|\bm{\omega}|)^{-x-N}\\[5pt]
&=
\dfrac{k_1!k_2!\cdots k_s!}{(x+n_1)^{k_1} (x+n_2)^{k_2}\cdots (x+n_s)^{k_s}}
\sum_{K_1,\ldots,K_s}
a^{n_1}(a+|\bm{\omega}_{K_1}|)^{n_2-n_1}\\
&\qquad \times
\cdots \times (a+|\bm{\omega}_{K_1\sqcup\cdots\sqcup K_{s-1}}|)^{n_s-n_{s-1}}
(a+|\bm{\omega}|)^{N-n_s}
\cdot (a+|\bm{\omega}|)^{-x-N}\\[3pt]
&=
k_1!\cdots k_s!
\sum_{K_1,\ldots,K_s}
\dfrac{\bigl(\frac{a}{a+|\bm{\omega}_{K_1}|}\bigr)^{n_1}
\bigl(\frac{a+|\bm{\omega}_{K_1}|}{a+|\bm{\omega}_{K_1\sqcup K_2}|}\bigr)^{n_2}
\cdots \bigl(\frac{a+|\bm{\omega}_{K_1\sqcup \cdots \sqcup K_{s-1}}|}{a+|\bm{\omega}|}\bigr)^{n_s}}
{(x+n_1)^{k_1} (x+n_2)^{k_2}\cdots (x+n_s)^{k_s}}
\cdot (a+|\bm{\omega}|)^{-x}.
\end{align*}
For each $s\in \{1,\ldots,r\}$,
we consider the following map defined by the above correspondence:
\begin{align*}
&\{(j_1,\ldots,j_N) \in (\Zb_{\ge 0})^N \mid
j_1+\cdots +j_N=r,  \ \#\{n \mid 0\le n\le N-1, \ j_{n+1}\ne 0\}=s\}\\
&\ni (j_1,j_2,\ldots,j_N)\\
&\mapsto
(k_1,k_2,\ldots,k_s;n_1,n_2,\ldots,n_s)\\
&\quad \in 
\left\{
(k_1,\ldots,k_s;n_1,\ldots,n_s)
\in (\Zb_{\ge 1})^s\times (\Zb_{\ge 0})^s
\, \middle| \,
\begin{array}{l}
k_1+\cdots+k_s=r\\
0\le n_1<n_2<\cdots <n_s< N
\end{array}
\right\}.
\end{align*}
Since this map is a bijection, we see
\begin{align*}
&(a+|\bm{\omega}|)^xT_N(r)\\
&=\sum_{s=1}^r
\sum_{\substack{\bm{k}\in (\Zb_{\ge 1})^s\\ k_1+\cdots+k_s=r}}
k_1!\cdots k_s!\\
&\quad \times 
\sum_{K_1,\ldots,K_s}
\sum_{0\le n_1<n_2<\cdots <n_s< N}
\dfrac{\bigl(\frac{a}{a+|\bm{\omega}_{K_1}|}\bigr)^{n_1}
\bigl(\frac{a+|\bm{\omega}_{K_1}|}{a+|\bm{\omega}_{K_1\sqcup K_2}|}\bigr)^{n_2}
\cdots \bigl(\frac{a+|\bm{\omega}_{K_1\sqcup \cdots \sqcup K_{s-1}}|}{a+|\bm{\omega}|}\bigr)^{n_s}}
{(x+n_1)^{k_1} (x+n_2)^{k_2}\cdots (x+n_s)^{k_s}}.
\end{align*}
Taking the limit as $N\rightarrow \infty$,
we obtain \eqref{eq:relation between I and Li0}.
\end{proof}

To prove Theorem \ref{power series for I},
it is more convenient to express $(a+|\bm{\omega}|)^xI_r(x;\bm{\omega},a)$ as a finite sum of functions of the form \eqref{eq:Li1}, as follows.
\begin{proposition}
\label{I-Li1}
\begin{align}
\label{eq:relation between I and Li1}
&(a+|\bm{\omega}|)^x
I_r(x;\bm{\omega},a)\\
&=
r!x^{-r}+\sum_{t=1}^r\sum_{s=1}^t
\sum_{\substack{\bm{k}\in (\Zb_{\ge 1})^s \\ k_1+\cdots+k_s=t}}
 (r-t)!k_1!\dots k_s! x^{t-r} \notag\\
&\quad \times
\sum_{K_1,\ldots,K_s}
\Li_{\bm{k}}^1
\biggl(x;
\dfrac{a+|\bm{\omega}_{K_0}|}{a+|\bm{\omega}_{K_0\sqcup K_1}|},
\dfrac{a+|\bm{\omega}_{K_0\sqcup K_1}|}{a+|\bm{\omega}_{K_0\sqcup K_1\sqcup K_2}|},
\ldots,
\dfrac{a+|\bm{\omega}_{K_0\sqcup \cdots \sqcup K_{s-1}}|}
{a+|\bm{\omega}|}
\biggr), \notag
\end{align}
where $K_1,\ldots,K_s$ are disjoint subsets of $R$
satisfying $|K_i|=k_i$ for $i=1,2,\ldots,s$
and  $K_0=R\setminus \bigsqcup_{i=1}^s K_i$
\end{proposition}

\begin{proof}
It easily follows from \eqref{eq:Li0} and \eqref{eq:Li1} that
\begin{equation}
\Li_{\bm{k}}^0(x;\bm{z})
=
\begin{cases}
x^{-k_1} +\Li_{\bm{k}}^1(x;\bm{z})
& \mbox{if $r=1$}\\
x^{-k_1} \Li_{(k_2,\ldots,k_r)}^1 (x;z_2,\ldots,z_r)
+\Li_{\bm{k}}^1(x;\bm{z})
& \mbox{if $r\ge 2$}.
\end{cases}
\end{equation}
Hence, we can rewrite \eqref{eq:relation between I and Li0} as
\begin{align}
\label{eq:I-Li1}
&(a+|\bm{\omega}|)^xI_r(x;\bm{\omega},a)\\
&=
r! x^{-r}
+\sum_{s=2}^r \sum_{\substack{\bm{k} \in (\Zb_{\ge 1})^s\\ k_1+\cdots+k_s=r}}
k_1!\cdots k_s! \notag\\
&\qquad \times
\sum_{K_1,\ldots,K_s}
x^{-k_1}
\Li_{(k_2,\ldots,k_s)}^1\biggl(
x;\dfrac{a+|\bm{\omega}_{K_1}|}{a+|\bm{\omega}_{K_1\sqcup K_2}|},\ldots,
\dfrac{a+|\bm{\omega}_{K_1\sqcup\cdots \sqcup K_{s-1}}|}{a+|\bm{\omega}|}\biggr) \notag\\
&\quad +
\sum_{s=1}^r \sum_{\substack{\bm{k}\in (\Zb_{\ge 1})^s\\ k_1+\cdots+k_s=r}}
k_1!\cdots k_s!
\sum_{K_1,\cdots,K_s}
\Li_{\bm{k}}^1\biggl(
x;\dfrac{a}{a+|\bm{\omega}_{K_1}|},\ldots,
\dfrac{a+|\bm{\omega}_{K_1\sqcup\cdots\sqcup K_{s-1}}|}{a+|\bm{\omega}|}\biggr). \notag
\end{align}
Note that the third term on the right-hand side of \eqref{eq:I-Li1}
are equal to the second term when $t=r$ on the right-hand side of \eqref{eq:relation between I and Li1} (because in this case $K_0=\emptyset$).
Therefore, it suffices to show that the second term, denoted as $S$, on the right-hand side of \eqref{eq:I-Li1}
is the same as the second term restricted by $t\le r-1$ on the right-hand side of \eqref{eq:relation between I and Li1}.
For $S$,
set $t=k_1$, $\bm{l}=(l_1,\ldots,l_{s-1})=(k_2,\ldots,k_s)$,
and $L_i=K_{i+1}$ for $i=0,1,\ldots,s-1$.
Then clearly $L_0=R\setminus \bigsqcup_{i=1}^{s-1} L_i$ and
\begin{align*}
S&=
\sum_{s=2}^r \sum_{t=1}^{r-s+1}
\sum_{\substack{\bm{l} \in (\Zb_{\ge 1})^{s-1}\\l_1+\cdots+l_{s-1}=r-t}}
t! l_1!\cdots l_{s-1}! x^{-t}\\
&\quad \times
\sum_{L_1,\ldots,L_{s-1}}
\Li_{\bm{l}}^1\biggl(
x;\dfrac{a+|\bm{\omega}_{L_0}|}{a+|\bm{\omega}_{L_0\sqcup L_1}|},
\ldots,
\dfrac{a+|\bm{\omega}_{L_0\sqcup\cdots\sqcup L_{s-2}}|}{a+|\bm{\omega}|}
\biggr),
\end{align*}
where $L_1,\ldots,L_{s-1}$ are disjoint subsets of $R$
satisfying $|L_i|=l_i$ for $i=1,\ldots,s-1$.
Moreover, by putting $u=s-1$ and $v=r-t$,
we have
\begin{align*}
S&=
\sum_{u=1}^{r-1} \sum_{v=u}^{r-1}
\sum_{\substack{\bm{l} \in (\Zb_{\ge 1})^{u}\\l_1+\cdots+l_u=v}}
(r-v)! l_1!\cdots l_u! x^{v-r}\\
&\quad \times
\sum_{L_1,\ldots,L_u}
\Li_{\bm{l}}^1\biggl(
x;\dfrac{a+|\bm{\omega}_{L_0}|}{a+|\bm{\omega}_{L_0\sqcup L_1}|},
\ldots,
\dfrac{a+|\bm{\omega}_{L_0\sqcup\cdots\sqcup L_{u-1}}|}{a+|\bm{\omega}|}
\biggr).
\end{align*}
By interchanging the sums over $u$ and $v$,
we see that $S$ coincides with the second term restricted by $t\le r-1$ on the right-hand side of \eqref{eq:relation between I and Li1}.
This completes the proof of \eqref{eq:relation between I and Li1}.
\end{proof}

To derive  \eqref{eq:power series for I} from \eqref{eq:relation between I and Li1},
we need the following
\begin{lemma}
\label{Li1-Li}
Assume $\bm{k}\in (\Zb_{\ge 1})^r$.
If $0<x<1$, then
\begin{equation}
\label{eq:Li1-Li}
\Li_{\bm{k}}^1(x;\bm{z})
=\sum_{l=0}^\infty (-x)^l
\sum_{\substack{\bm{l}\in(\Zb_{\ge 0})^r\\ l_1+\cdots+l_r=l}}
\prod_{i=1}^r \binom{k_i+l_i-1}{l_i}
\cdot \Li_{\bm{k}+\bm{l}}(\bm{z}).
\end{equation}
\end{lemma}

\begin{proof}
Since
\[
\dfrac{1}{(n+x)^k}
=\dfrac{1}{n^k}\cdot \dfrac{1}{(1+x/n)^k}
=\dfrac{1}{n^k} \sum_{l=0}^\infty \binom{-k}{l} \left(\dfrac{x}{n}\right)^l
=\sum_{l=0}^\infty
\binom{k+l-1}{l} \dfrac{(-x)^l}{n^{k+l}}
\]
for $n\ge 1$, we have
\[
\Li_{\bm{k}}^1 (x;\bm{z})
=\sum_{1\le n_1<\cdots <n_r}
\sum_{l_1,\ldots,l_r=0}^\infty
\prod_{i=1}^r \binom{k_i+l_i-1}{l_i}
\cdot
\dfrac{(-x)^{l_1+\cdots +l_r}z_1^{n_1}\cdots z_r^{n_r}}
{n_1^{k_1+l_1}\cdots n_r^{k_r+l_r}}.
\]
This shows \eqref{eq:Li1-Li}.
\end{proof}

Finally, we prove Theorem \ref{power series for I}.
By Proposition  \ref{I-Li1}
and Lemma  \ref{Li1-Li},
we have
\begin{align*}
&(a+|\bm{\omega}|)^x I_r(x;\bm{\omega},a)\\
&=r!x^{-r}+
\sum_{l=0}^\infty
\sum_{t=1}^r
(-1)^l x^{t+l-r}\\
&\quad \times
\sum_{s=1}^t
\sum_{\substack{\bm{k}\in (\Zb_{\ge1})^s\\ k_1+\cdots+k_s=t}}
\sum_{\substack{\bm{l}\in (\Zb_{\ge 0})^s\\ l_1+\cdots+l_s=l}}
(r-t)! k_1!\cdots k_s!
\prod_{i=1}^s \binom{k_i+l_i-1}{l_i}\\
&\quad
\times
\sum_{K_1,\ldots,K_s}
\Li_{\bm{k}+\bm{l}}\biggl(
\dfrac{a+|\bm{\omega}_{K_0}|}{a+|\bm{\omega}_{K_0\sqcup K_1}|},
\ldots,
\dfrac{a+|\bm{\omega}_{K_0\sqcup\cdots\sqcup K_{s-1}}|}{a+|\bm{\omega}|}
\biggr).
\end{align*}
By setting $m=t+l$ and rewriting the sum over $l$
in terms of a sum over $m$,
we obtain Theorem \ref{power series for I}.

\section{Application to multiple polylogarithms}
\label{sec:application to MPL}
In this section, we establish Theorems \ref{relation among MPL} and \ref{v9 relation between coefficients c and d}, which give some nontrivial relations among multiple polylogarithms.

We begin to see that the first factors of the main terms in \eqref{eq:asymptotic for I} and \eqref{eq:asymptotic for M}
have the following Taylor expansions at $x=0$, whose coefficients can be expressed in terms of the complete exponential Bell polynomials $B_k(x_1,\ldots,x_k)$
defined by \eqref{recurrence relation for Bell polynomials}.

\begin{proposition}
For $x\in\Cb$, we have
\begin{equation}
\label{power series for Gamma 1}
\dfrac{e^{-\gamma x}}{\Gamma(x+1)}
=\sum_{k=0}^\infty
B_k(0,-1!\zeta(2),2!\zeta(3),\ldots,(-1)^{k-1}(k-1)!\zeta(k))
\dfrac{x^k}{k!}
\end{equation}
and
\begin{equation}
\label{power series for Gamma 2}
\dfrac{1}{\Gamma(x+1)}
=\sum_{k=0}^\infty
B_k(\gamma,-1!\zeta(2),2!\zeta(3),\ldots,(-1)^{k-1}(k-1)!\zeta(k))
\dfrac{x^k}{k!}.
\end{equation}
\end{proposition}

\begin{proof}
By \eqref{eq:generating function} and \eqref{eq:power series of log gamma}, we see that 
\[
\dfrac{e^{-\gamma x}}{\Gamma(1+x)}
=
\sum_{n=0}^\infty 
B_n(0,-1!\zeta(2),2!\zeta(3),\ldots,(-1)^{n-1}(n-1)! \zeta(n))
\dfrac{x^n}{n!}
\]
when $0<x<1$.
Since $e^{-\gamma x}/\Gamma(x+1)$ is an entire function,
\eqref{power series for Gamma 1} holds for $x\in\Cb$.
Lastly, we multiply the both sides of \eqref{eq:power series of log gamma}
by $e^{\gamma x}$ and argue similarly to obtain
\eqref{power series for Gamma 2}. 
\end{proof}



We now prove Theorem \ref{relation among MPL}. 
\begin{proof}[Proof of Theorem \ref{relation among MPL}]
The value
$c_{r,m}(\bm{\omega},a)$ is the coefficient of $x^{m-r}$ in the power series of
$(a+|\bm{\omega}|)^x I_r(x;\bm{\omega},a)$.
To evaluate it from \eqref{eq:asymptotic for I},
we first note the relation
\[
(a+|\bm{\omega}|)^x I_r(x;\bm{\omega},a)
=I_r\biggl(x;\dfrac{\bm{\omega}}{a+|\bm{\omega}|},\dfrac{a}{a+|\bm{\omega}|}\biggr),
\]
which immediately follows from the definition of $I_r(x;\bm{\omega},a)$.
Now, applying \eqref{eq:asymptotic for I} to the integral on the right-hand side of the above, we obtain
\[
(a+|\bm{\omega}|)^x I_r(x;\bm{\omega},a)
=\dfrac{e^{-\gamma x}}{\Gamma(x+1)}
\sum_{l=0}^r (-1)^l (r-l)! \Lambda_l\left(\dfrac{\bm{\omega}}{a+|\bm{\omega}|}\right)x^{l-r}
+O(x).
\]
Combining this asymptotic expansion and \eqref{power series for Gamma 1},
we have
\begin{align*}
&(a+|\bm{\omega}|)^x I_r(x;\bm{\omega},a)\\
&=\sum_{k=0}^{\infty}B_k(0,-1!\zeta(2),2!\zeta(3),\ldots,(-1)^{k-1}(k-1)!\zeta(k))
\frac{x^k}{k!}\\
&\qquad\times\sum_{l=0}^r (-1)^l (r-l)! \Lambda_l\left(\dfrac{\bm{\omega}}{a+|\bm{\omega}|}\right)x^{l-r}+O(x)\\
&=\sum_{m=0}^{\infty} x^{m-r}
\sum_{k=0}^m (-1)^{m-k}\dfrac{(r-m+k)!}{k!}
\Lambda_{m-k}\left(\dfrac{\bm{\omega}}{a+|\bm{\omega}|}\right)\\
&\qquad \times
B_k(0,-1!\zeta(2),2!\zeta(3),\ldots,(-1)^{k-1}(k-1)!\zeta(k))\\
&\quad +O(x)\\
&=\sum_{m=0}^r x^{m-r}
\sum_{k=0}^m (-1)^{m-k}\dfrac{(r-m+k)!}{k!}
\Lambda_{m-k}\left(\dfrac{\bm{\omega}}{a+|\bm{\omega}|}\right)\\
&\qquad \times
B_k(0,-1!\zeta(2),2!\zeta(3),\ldots,(-1)^{k-1}(k-1)!\zeta(k))\\
&\quad +O(x).
\end{align*}
This establishes Theorem \ref{relation among MPL}.
\end{proof}

We notice that \eqref{eq:relation among MPL} apparently produces
infinitely many relations among multiple polylogarithms of fixed weight $m$.
However, the case $r=m$ in \eqref{eq:relation among MPL} is essential.
In fact, the other relations follow from this case, as is easily seen from the following

\begin{proposition}
Let $c'_{r,m}(\bm{\omega},a)$ denote the right-hand side of \eqref{eq:relation among MPL}.
Then, without using \eqref{eq:relation among MPL}, we can independently prove that
\begin{equation}
\label{eq:recurrent relation for c}
c_{r,m}(\bm{\omega},a)=\sum_{i=1}^r c_{r-1,m}(\bm{\omega}_{R\setminus \{i\}},a+\omega_i)
\end{equation}
and
\begin{equation}
\label{eq:recurrent relation for c'}
c'_{r,m}(\bm{\omega},a)=\sum_{i=1}^r c'_{r-1,m}(\bm{\omega}_{R\setminus \{i\}},a+\omega_i)
\end{equation}
when $r\ge m+1$.
\end{proposition}

\begin{proof}
To prove \eqref{eq:recurrent relation for c},
we evaluate its right-hand side by using \eqref{eq:def of c}.
Then we encounter the sum
\[
\sum_{i=1}^r \sum_{K_1^i,\ldots,K_s^i}
\Li_{\bm{k}+\bm{l}}
\biggl(
\dfrac{a+|\bm{\omega}_{K_0^i\sqcup \{i\}}|}{a+|\bm{\omega}_{(K_0^i\sqcup\{i\})\sqcup K_1^i}|},
\ldots,\dfrac{a+|\bm{\omega}_{(K_0^i\sqcup \{i\})\sqcup\cdots \sqcup K_{s-1}^i}|}{a+|\bm{\omega}|}
\biggr)
\]
for $\bm{k}\in (\Zb_{\ge 1})^s$ with $k_1+\cdots +k_s=t$
and $\bm{l}\in (\Zb_{\ge 0})^s$ with $l_1+\cdots +l_s=m-t$,
where $K_1^i,\ldots,K_s^i$ are disjoint subsets of $R\setminus \{i\}$
with $|K_j^i|=k_j$
and $K_0^i=R\setminus (\{i\}\sqcup \bigsqcup_{j=1}^s K_j^i)$ for each $i$.
This sum can be transformed as
\[
\sum_{K_1,\ldots,K_s}
\Li_{\bm{k}+\bm{l}}
\biggl(
\dfrac{a+|\bm{\omega}_{K_0}|}{a+|\bm{\omega}_{K_0\sqcup K_1}|},
\ldots,\dfrac{a+|\bm{\omega}_{K_0 \sqcup\cdots \sqcup K_{s-1}}|}{a+|\bm{\omega}|}
\biggr)
\sum_{\substack{1\le i\le r\\i\not\in K_1\sqcup \cdots \sqcup K_s}} 1,
\]
where $K_1,\ldots,K_s$ are disjoint subsets of $R$
with $|K_j|=k_j$
and $K_0=R\setminus \bigsqcup_{j=1}^s K_j$.
Since $|K_1\sqcup \cdots \sqcup K_s|=k_1+\cdots +k_s=t$,
the inner sum coincides with $r-t$.
We can easily check that \eqref{eq:recurrent relation for c} follows from this transformation.

In a similar fashion,
\eqref{eq:recurrent relation for c'} can be seen
by the definition of $c_{r,m}'(\bm{\omega},a)$ and the following transformation:
\begin{align*}
\sum_{i=1}^r \Lambda_{m-k}\biggl(\dfrac{\bm{\omega}_{R\setminus\{i\}}}{a+|\bm{\omega}|}\biggr)
&=
\sum_{i=1}^r \sum_{\substack{J\subset R\setminus\{i\} \\ |J|=m-k}}
\prod_{j\in J}\log \dfrac{\omega_j}{a+|\bm{\omega}|}\\
&=
\sum_{\substack{J\subset R \\ |J|=m-k}}
\biggl(\prod_{j\in J}\log \dfrac{\omega_j}{a+|\bm{\omega}|}\biggr)
\sum_{\substack{1\le i\le r\\ i\not\in J}} 1\\
&=
(r-m+k)
\Lambda_{m-k}\biggl(\dfrac{\bm{\omega}}{a+|\bm{\omega}|}\biggr).
\qedhere
\end{align*}
\end{proof}
\par
Next, we present the proof of Theorem \ref{v9 relation between coefficients c and d}.
\begin{proof}[Proof of Theorem \ref{v9 relation between coefficients c and d}]
	Using Theorem \ref{power series for I}, we have
	\begin{align*}
		\lefteqn{a^x I_r(x;\bm{\omega},a)}
		\quad & \\
		&=\biggl(\dfrac{a}{a+|\bm{\omega}|}\biggr)^x
		(a+|\bm{\omega}|)^x I_r(x;\bm{\omega},a)\\
		&=
		\biggl(\sum_{m=0}^\infty
		\biggl(\log \dfrac{a}{a+|\bm{\omega}|}\biggr)^m
		\dfrac{x^m}{m!}\biggr)
		\biggl(r!x^{-r}+\sum_{k=1}^\infty c_{r,k}(\bm{\omega},a)x^{k-r}\biggr)\\
		&=r!\sum_{m=0}^\infty
		\biggl(\log \dfrac{a}{a+|\bm{\omega}|}\biggr)^m
		\dfrac{x^{m-r}}{m!}
		+\sum_{m=1}^\infty
		\biggl(\sum_{k=1}^m \biggl(\log \dfrac{a}{a+|\bm{\omega}|}\biggr)^{m-k}
		\dfrac{c_{r,k}(\bm{\omega},a)}{(m-k)!}\biggr)
		x^{m-r}.
	\end{align*}
	Then identity \eqref{v9 d in terms c} follows from Theorem \ref{v9 first complete asymptotic expansion of integral}.
	\vspace{2mm}\\
	Similarly with the help of Theorem  \ref{v9 first complete asymptotic expansion of integral}, we express
	\begin{align*}
		\lefteqn{(a+|\bm{\omega}|)^x I_r(x;\bm{\omega},a)}
		\quad & \\
		&=\biggl(\dfrac{a+|\bm{\omega}|}{a}\biggr)^{x} a^x I_r(x;\bm{\omega},a)\\
		&=
		\biggl(\sum_{m=0}^\infty
		\biggl(\log \dfrac{a+|\bm{\omega}|}{a}\biggr)^m
		\dfrac{x^m}{m!}\biggr)
		\biggl(r!x^{-r}+\sum_{k=1}^\infty d_{r,k}(\bm{\omega},a)x^{k-r}\biggr)\\
		&=r!\sum_{m=0}^\infty
		\biggl(\log \dfrac{a+|\bm{\omega}|}{a}\biggr)^m
		\dfrac{x^{m-r}}{m!}
		+\sum_{m=1}^\infty
		\biggl(\sum_{k=1}^m \biggl(\log \dfrac{a+|\bm{\omega}|}{a}\biggr)^{m-k}
		\dfrac{d_{r,k}(\bm{\omega},a)}{(m-k)!}\biggr)
		x^{m-r},
	\end{align*}
and then apply Theorem \ref{power series for I} to get \eqref{v9 c in terms d}.
	\end{proof}

We display some examples of \eqref{eq:relation among MPL} for $r=m$.
\begin{example}
$(1)$ The case $r=m=1$:
\eqref{eq:relation among MPL} is the trivial relation
\[
\Li_1\biggl(\dfrac{a}{a+\omega}\biggr)
=-\log \dfrac{\omega}{a+\omega}.
\]

$(2)$ The case $r=m=2$:
\eqref{eq:relation among MPL} agrees with \eqref{eq:example for r=m=2}.

$(3)$ The case $r=m=3$:
\eqref{eq:relation among MPL} can be expressed as
\begin{align*}
&2\sum_{i=1,2,3}
\Li_3\biggl(\dfrac{a+|\bm{\omega}|-\omega_i}{a+|\bm{\omega}|}\biggr)
-4\sum_{i=1,2,3} \Li_3\biggl(\dfrac{a+\omega_i}{a+|\bm{\omega}|}\biggr)\\[5pt]
&\quad
-\sum_{\sigma\in \mathfrak{S}_3}
\Li_{1,2}\biggl(\dfrac{a+\omega_{\sigma(1)}}{a+\omega_{\sigma(1)}+\omega_{\sigma(2)}},
\dfrac{a+\omega_{\sigma(1)}+\omega_{\sigma(2)}}{a+|\bm{\omega}|}\biggr)\\[5pt]
&\quad
-\sum_{\sigma\in \mathfrak{S}_3}
\Li_{2,1}\biggl(\dfrac{a+\omega_{\sigma(1)}}{a+\omega_{\sigma(1)}+\omega_{\sigma(2)}},
\dfrac{a+\omega_{\sigma(1)}+\omega_{\sigma(2)}}{a+|\bm{\omega}|}\biggr)\\[5pt]
&\quad
+6\Li_3\biggl(\dfrac{a}{a+|\bm{\omega}|}\biggr)
+2\sum_{i=1,2,3}
\Li_{1,2}\biggl(\dfrac{a}{a+\omega_i},\dfrac{a+\omega_i}{a+|\bm{\omega}|}\biggr)\\
&\quad
+2\sum_{i=1,2,3}
\Li_{2,1}\biggl(\dfrac{a}{a+|\bm{\omega}|-\omega_i},\dfrac{a+|\bm{\omega}|-\omega_i}{a+|\bm{\omega}|}\biggr)\\[5pt]
&\quad
+\sum_{\sigma \in \mathfrak{S}_3}
\Li_{1,1,1}\biggl(
\dfrac{a}{a+\omega_{\sigma(1)}},
\dfrac{a+\omega_{\sigma(1)}}{a+\omega_{\sigma(1)}+\omega_{\sigma(2)}},
\dfrac{a+\omega_{\sigma(1)}+\omega_{\sigma(2)}}{a+|\bm{\omega}|}\biggr)\\
&=
-\log \dfrac{\omega_1}{a+|\bm{\omega}|}
\log \dfrac{\omega_2}{a+|\bm{\omega}|}
\log \dfrac{\omega_3}{a+|\bm{\omega}|}
+\dfrac{\pi^2}{6}
\log \dfrac{\omega_1\omega_2\omega_3}
{(a+|\bm{\omega}|)^3}
+2\zeta(3),
\end{align*}
where $\mathfrak{S}_3$ is the symmetric group of degree three.
\end{example}


Now we demonstrate the applications of  Theorem \ref{v9 relation between coefficients c and d}. 

\begin{example}
$(1)$ The case $r=1$:
the corresponding expressions agree with those in Corollary \ref{cor:1,...,1,2}.

$(2)$ The case $r=2$, $m=2$:
\eqref{v9 c in terms d} reads
\begin{align*}
&-\Li_2 \Bigl(\dfrac{a+\omega_1}{a+\omega_1+\omega_2}\Bigr)
-\Li_2 \Bigl(\dfrac{a+\omega_2}{a+\omega_1+\omega_2}\Bigr)
+2\Li_2 \Bigl(\dfrac{a}{a+\omega_1+\omega_2}\Bigr)\\
&
+\Li_{1,1}\Bigl(\dfrac{a}{a+\omega_1},\dfrac{a+\omega_1}{a+\omega_1+\omega_2}\Bigr)
+\Li_{1,1}\Bigl(\dfrac{a}{a+\omega_2},\dfrac{a+\omega_2}{a+\omega_1+\omega_2}\Bigr)\\
&=
\Bigl(\log \dfrac{a+\omega_1+\omega_2}{a}\Bigr)^2
-\Bigl(\log \dfrac{a+\omega_1+\omega_2}{a}\Bigr) \log \dfrac{\omega_1\omega_2}{a^2}
+\log \dfrac{\omega_1}{a} \log \dfrac{\omega_2}{a} -\dfrac{\pi^2}{6}.
\end{align*}
This right-hand side can be simplified to
\[
\log \dfrac{\omega_1}{a+\omega_1+\omega_2} \cdot \log \dfrac{\omega_2}{a+\omega_1+\omega_2} -\dfrac{\pi^2}{6}.
\]
Hence this relation coincides with \eqref{eq:example for r=m=2}.

$(3)$ The case $r=2$, $m=3$:
\eqref{v9 c in terms d} reads
\begin{align*}
&\Li_3\Bigl(\dfrac{a+\omega_1}{a+\omega_1+\omega_2}\Bigr)
+\Li_3\Bigl(\dfrac{a+\omega_2}{a+\omega_1+\omega_2}\Bigr)
-4\Li_3\Bigl(\dfrac{a}{a+\omega_1+\omega_2}\Bigr)\\
& - \Li_{1,2} \Bigl(\dfrac{a}{a+\omega_1},\dfrac{a+\omega_1}{a+\omega_1+\omega_2}\Bigr)
-\Li_{1,2} \Bigl(\dfrac{a}{a+\omega_2},\dfrac{a+\omega_2}{a+\omega_1+\omega_2}\Bigr)\\
& - \Li_{2,1} \Bigl(\dfrac{a}{a+\omega_1},\dfrac{a+\omega_1}{a+\omega_1+\omega_2}\Bigr)
-\Li_{2,1} \Bigl(\dfrac{a}{a+\omega_2},\dfrac{a+\omega_2}{a+\omega_1+\omega_2}\Bigr)\\
&=\dfrac{1}{3} \Bigl(\log \dfrac{a+\omega_1+\omega_2}{a}\Bigr)^3
-\dfrac{1}{2} \Bigl(\log \dfrac{a+\omega_1+\omega_2}{a}\Bigr)^2 \log \dfrac{\omega_1\omega_2}{a}\\
&\quad
+\Bigl(\log \dfrac{a+\omega_1+\omega_2}{a}\Bigr)
\Bigl(\log \dfrac{\omega_1}{a}\log \dfrac{\omega_2}{a}-\dfrac{\pi^2}{6}\Bigr)\\
&\quad
+\dfrac{\pi^2}{6} \log \dfrac{\omega_1\omega_2}{a^2}
+T_{2,1}\biggl(\Bigl(-\dfrac{\omega_1}{a},-\dfrac{\omega_2}{a}\Bigr)\biggr)\\
&\quad
+\log \dfrac{\omega_1}{a} \Li_2\Bigl(-\dfrac{\omega_2}{a}\Bigr)
+\log \dfrac{\omega_2}{a} \Li_2\Bigl(-\dfrac{\omega_1}{a}\Bigr)
+\Li_{1,2}\Bigl(-\dfrac{\omega_1}{a}\Bigr)
+\Li_{1,2}\Bigl(-\dfrac{\omega_2}{a}\Bigr).
\end{align*}
\end{example}

We obtain various formulas among polylogarithms in this section, but it is not clear
whether those formulas are really new, or essentially already known.    In references 
we cannot find
the general form of Theorem \ref{relation among MPL}, so probably it is new.   However, 
we can see that some special cases of Theorem \ref{relation among MPL} can indeed be deduced
from known formulas.   We conclude the present paper with giving a sketch of such deduction of 
\eqref{eq:example for r=m=2} and of \eqref{cor-1-1}.

First we deduce \eqref{eq:example for r=m=2} from the following three formulas.
We have
\begin{equation}\label{eq:A}
\Li_{1,1}(x,y)
=\Li_2\!\left(\frac{y(1-x)}{y-1}\right)-\Li_2\!\left(\frac{y}{y-1}\right)-\Li_2(xy)
\end{equation}
for $|x|,|y|<1$,
\begin{equation}\label{eq:B}
\Li_2\!\left(\frac{z}{z-1}\right)=-\Li_2(z)-\frac12\log^2(1-z)
\end{equation}
for $z<1$,
and 
\begin{equation}\label{eq:C}
\Li_2(-z)+\Li_2\!\left(-\frac1z\right)=-\frac{\pi^2}{6}-\frac12\log^2 z
\end{equation}
for $z>0$.
Here, the second formula is \cite[(1.12)]{Lewin}, and the third formula is \cite[(1.7)]{Lewin}
(with the fact $\Li_2(-1)=-\pi^2/12$).
To prove the first formula \eqref{eq:A}, define
\[
F(x)
:=
\Li_{1,1}(x,y)
-
\Li_2\!\left(\frac{y(1-x)}{y-1}\right)
+
\Li_2\!\left(\frac{y}{y-1}\right)
+
\Li_2(xy)
\]
and compute $\partial F(x)/\partial x$.
Since
\begin{align*}
\frac{\partial}{\partial x}\Li_{1,1}(x,y)
&=
\sum_{n=2}^\infty \frac{y^n}{n}\sum_{m=1}^{n-1}x^{m-1}
=
\sum_{n=2}^\infty \frac{y^n}{n}\frac{1-x^{n-1}}{1-x}\\
&=
\frac{1}{1-x}
\left(
\frac{\log(1-xy)}{x}
-
\log(1-y)
\right),
\end{align*}
\[
\frac{\partial}{\partial x}\Li_2(xy)
=
-\frac{\log(1-xy)}{x},
\]
and
\[
\frac{\partial}{\partial x}\Li_2\left( \frac{y(1-x)}{y-1}\right)
=
\frac{1}{1-x}\bigl(\log(1-xy)-\log(1-y)\bigr),
\]
we find that
$\partial F(x)/\partial x=0$
for $|x|<1$.
Thus $F(x)$ is a constant function with respect to $x$. 
Evaluating at $x=0$, we obtain
\[
F(0)
=
\Li_{1,1}(0,y)
-
\Li_2\!\left(\frac{y}{y-1}\right)
+
\Li_2\!\left(\frac{y}{y-1}\right)
+
\Li_2(0)
=
0
\]
for any $y$.
Hence $F(x)\equiv 0$ for $|x|<1$, which proves \eqref{eq:A}.

Set $S:=a+\omega_1+\omega_2$ and
\[
x_1:=\frac{a}{a+\omega_1},\quad y_1:=\frac{a+\omega_1}{S},\qquad
x_2:=\frac{a}{a+\omega_2},\quad y_2:=\frac{a+\omega_2}{S}.
\]
Since
\[
\frac{y_1(1-x_1)}{y_1-1}
=\frac{\frac{a+\omega_1}{S}\cdot\frac{\omega_1}{a+\omega_1}}{-\omega_2/S}
=-\frac{\omega_1}{\omega_2},
\qquad
\frac{y_1}{y_1-1}=-\frac{a+\omega_1}{\omega_2},
\]
and $x_1y_1=a/S$, \eqref{eq:A} gives
\[
\Li_{1,1}(x_1,y_1)
=\Li_2\!\left(-\frac{\omega_1}{\omega_2}\right)
-\Li_2\!\left(-\frac{a+\omega_1}{\omega_2}\right)
-\Li_2\!\left(\frac{a}{S}\right),
\]
and similarly,
\[
\Li_{1,1}(x_2,y_2)
=\Li_2\!\left(-\frac{\omega_2}{\omega_1}\right)
-\Li_2\!\left(-\frac{a+\omega_2}{\omega_1}\right)
-\Li_2\!\left(\frac{a}{S}\right).
\]
Therefore the left-hand side of \eqref{eq:example for r=m=2} is
\begin{align}\label{AAA}
=-\Li_2(y_1)-\Li_2(y_2)
+\Li_2\!\left(-\frac{\omega_1}{\omega_2}\right)+\Li_2\!\left(-\frac{\omega_2}{\omega_1}\right)
-\Li_2\!\left(-\frac{a+\omega_1}{\omega_2}\right)-\Li_2\!\left(-\frac{a+\omega_2}{\omega_1}\right),
\end{align}
because the $\Li_2(a/S)$ terms cancel.

Next, since
\[
\frac{y_1}{y_1-1}=-\frac{a+\omega_1}{\omega_2},\qquad
\frac{y_2}{y_2-1}=-\frac{a+\omega_2}{\omega_1},
\]
from \eqref{eq:B} with $z=y_1$ and $z=y_2$ we find
\[
-\Li_2\!\left(-\frac{a+\omega_1}{\omega_2}\right)=\Li_2(y_1)+\frac12\log^2(1-y_1),
\quad
-\Li_2\!\left(-\frac{a+\omega_2}{\omega_1}\right)=\Li_2(y_2)+\frac12\log^2(1-y_2),
\]
and we see that
\[
1-y_1=1-\frac{a+\omega_1}{S}=\frac{\omega_2}{S},\qquad
1-y_2=1-\frac{a+\omega_2}{S}=\frac{\omega_1}{S}.
\]
Substituting these formulas into \eqref{AAA}, we obtain that the left-hand side of
\eqref{eq:example for r=m=2} is
\begin{align}\label{BBB}
=\Li_2\!\left(-\frac{\omega_1}{\omega_2}\right)+\Li_2\!\left(-\frac{\omega_2}{\omega_1}\right)
+\frac12\log^2\!\left(\frac{\omega_2}{S}\right)+\frac12\log^2\!\left(\frac{\omega_1}{S}\right).
\end{align}
Further, \eqref{eq:C} with $z=\omega_1/\omega_2$ implies
\[
\Li_2\!\left(-\frac{\omega_1}{\omega_2}\right)+\Li_2\!\left(-\frac{\omega_2}{\omega_1}\right)
=-\frac{\pi^2}{6}-\frac12\log^2\!\left(\frac{\omega_1}{\omega_2}\right).
\]
and so, \eqref{BBB} is 
\begin{align*}
=-\frac{\pi^2}{6}
-\frac12\log^2\!\left(\frac{\omega_1}{\omega_2}\right)
+\frac12\log^2\!\left(\frac{\omega_1}{S}\right)+\frac12\log^2\!\left(\frac{\omega_2}{S}\right)
=\log(\omega_1/S)\log(\omega_2/S)-\pi^2/6
\end{align*}
which is exactly the right-hand side of \eqref{eq:example for r=m=2}.

Finally we show that \eqref{cor-1-1} can be easily deduced from known formulas, proved in
Kaneko and Tsumura \cite{KanekoTsumura}.

In fact, the formula with $k=2$ in \cite[Remark 3.7]{KanekoTsumura} is
\begin{align}
\Li_{\underbrace{\mbox{\scriptsize $1,\ldots,1$}}_{r-1},2}(1-z)
=
-\!\!\sum_{\substack{j_1+j_2=r+2\\ j_1,j_2\ge 1}}
\Li_{\underbrace{\mbox{\scriptsize $1,\ldots,1$}}_{j_2-1}}(1-z)\,\Li_{j_1}(z)
+\z(\onen{r-1},2)
\label{KT}
\end{align}
for $0<z<1$ and $r\in\mathbb{Z}_{\ge 1}$.
Noting
\[
\Li_{\underbrace{\mbox{\scriptsize $1,\ldots,1$}}_{j_2-1}}(w)=\frac{(-\log(1-w))^{j_2-1}}{(j_2-1)!}
\]
(\cite[Lemma 1 (ii)]{AK99}),
and $\z(\onen{r-1},2)=\zeta(r+1)$ (an instance of the duality), we can rewrite \eqref{KT} as
\begin{align}\label{CCC}
\Li_{\underbrace{\mbox{\scriptsize $1,\ldots,1$}}_{r-1},2}(1-z)
=
-\sum_{j=0}^{r}\frac{(-\log z)^{r-j}}{(r-j)!}\,\Li_{j+1}(z)
+\z(r+1).
\end{align}
Moreover, setting $z=1/(x+1)$ we see that \eqref{CCC} can be written as
\begin{align*}
\label{KT2}
&\Li_{\underbrace{\mbox{\scriptsize $1,\ldots,1$}}_{r-1},2}\biggl(\dfrac{x}{x+1}\biggr)\\
&=
(-1)^r \biggl(
\sum_{j=0}^{r}\frac{(-1)^{j+1}}{(r-j)!} \biggl(\log \dfrac{1}{x+1}\biggr)^{r-j}
\Li_{j+1}\biggl(\dfrac{1}{x+1}\biggr)
+(-1)^r \z(r+1)\biggr).
\end{align*}
Therefore \eqref{cor-1-1} can be deduced from \eqref{CCC} if the identity
\begin{equation}
\label{T}
\Li_{\underbrace{\mbox{\scriptsize $1,\ldots,1$}}_{r-1},2}\biggl(\dfrac{x}{1+x}\biggr)
=
(-1)^r\Li_{\underbrace{\mbox{\scriptsize $1,\ldots,1$}}_{r-1},2}(-x)
-\frac{(\log (1+x))^{r+1}}{(r+1)!}
\end{equation}
holds.
But this coincides with the formula in \cite[Lemma 3.1]{KanekoTsumura}
with $\mathbf{k}=(\underbrace{1,\ldots,1}_{r-1},2)$ and $z=-x$.
Therefore, \eqref{cor-1-1} follows from Lemma 3.1 and Remark 3.7 in \cite{KanekoTsumura}.


\section*{acknowledgement}
The authors would like to express their sincere gratitude to Professor A. Dixit for bringing the paper \cite{DSS24} to their attention, and to Professor Masanobu Kaneko for valuable comments.

\end{document}